\pdfoutput=1
\documentclass[11pt]{amsart}
\usepackage{amsmath}
\usepackage{amssymb}
\usepackage{harpoon}
\usepackage[pdftex]{graphicx}
\usepackage{esint}
\usepackage{relsize}
\usepackage[latin1]{inputenc}
\usepackage{xcolor}

\definecolor{brass}{rgb}{0.71, 0.65, 0.36}

\usepackage{tikz-cd}
\theoremstyle{plain}
\newtheorem{theorem}{Theorem}[section]
\newtheorem{lemma}[theorem]{Lemma}
\newtheorem{prop}[theorem]{Proposition}
\newtheorem{corollary}[theorem]{Corollary}

\newtheorem{proposition}[theorem]{Proposition}

\theoremstyle{definition}
\newtheorem{remark}[theorem]{Remark}

\numberwithin{equation}{section}

\def\sup{\operatorname{sup}}
\def\inf{\operatorname{inf}}

\def\sup{\operatorname{sup}}

\def\dist{\operatorname{dist}}

\theoremstyle{plain}

\numberwithin{equation}{section}

\allowdisplaybreaks

\DeclareMathOperator{\hess}{Hess}

\DeclareMathOperator{\tr}{tr}

\begin{document}
 
\title[Almost Splitting and the Topology of Black Holes]{A Bakry-\'Emery Almost Splitting Result With Applications to the Topology of Black Holes}

\author[Galloway]{Gregory J. Galloway}
\address{Department of Mathematics\\
University of Miami\\
Coral Gables, FL 33146, USA}
\email{galloway@math.miami.edu}

\author[Khuri]{Marcus A. Khuri}
\address{Department of Mathematics\\
Stony Brook University\\
Stony Brook, NY 11794, USA}
\email{khuri@math.sunysb.edu}

\author[Woolgar]{Eric Woolgar}
\address{Department of Mathematical and Statistical Sciences\\
and Theoretical Physics Institute\\
University of Alberta\\
Edmonton, AB, Canada T6G 2G1}
\email{ewoolgar@ualberta.ca}


\thanks{G. J. Galloway acknowledges the support of NSF Grant DMS-1710808. M. A. Khuri acknowledges the support of NSF Grant DMS-1708798, and Simons Foundation Fellowship 681443. E. Woolgar acknowledges the support of a Discovery Grant RGPIN-2017-04896 from the Natural Sciences and Engineering Research Council.}

\begin{abstract}
The almost splitting theorem of Cheeger-Colding is established in the setting of almost nonnegative generalized $m$-Bakry-\'{E}mery Ricci curvature, in which $m$ is positive and the associated vector field is not necessarily required to be the gradient of a function. In this context it is shown that with a diameter upper bound and volume lower bound, as well as control on the Bakry-\'{E}mery vector field, the fundamental group of such manifolds is almost abelian. Furthermore, extensions of well-known results concerning Ricci curvature lower bounds are given for generalized $m$-Bakry-\'{E}mery Ricci curvature. These include: the first Betti number bound of Gromov and Gallot, Anderson's finiteness of fundamental group isomorphism types, volume comparison, the Abresch-Gromoll inequality, and a Cheng-Yau gradient estimate. Finally, this analysis is applied to stationary vacuum black holes in higher dimensions to find that low temperature horizons must have limited topology, similar to the restrictions exhibited by (extreme) horizons of zero temperature.
\end{abstract}

\maketitle

\section{Introduction}
\label{sec1} \setcounter{equation}{0}
\setcounter{section}{1}

What are the possible topologies of stationary black holes? As we will see, a new approach involves the study of generalized $m$-Bakry-\'{E}mery Ricci curvature lower bounds. Let us recall previous techniques and results.
The topology of stationary black holes in 4-dimensional spacetime is tightly constrained by energy conditions. Hawking \cite{Hawking}, \cite[Proposition 9.3.2]{HE} proved that if the dominant energy condition holds, then apparent horizons of stationary black holes in $4$-dimensional spacetimes must have spherical topology; a borderline case that could have admitted toroidal topology was definitively eliminated more recently \cite{Galloway}. An independent theorem, based on the topological censorship theorem \cite{FSW} and requiring instead the null energy condition but also implying spherical horizon topology (in this case, for the event horizon itself) in $4$-dimensional stationary spacetimes, was first noticed in \cite{CW} and was generalized in \cite{GSWW}.

In higher dimensions, the situation is quite different. Although topological censorship applies in $5$ and more dimensions, it places no significant restrictions on event horizon topology. Hawking's theorem can be generalized to higher dimensions \cite{GS}, and implies that the horizon must be of positive Yamabe type, but this is a relatively mild restriction in higher dimensions. In $5$ spacetime dimensions it permits orientable horizon cross-sections to have the topology of spherical spaces, $S^1\times S^2$, or connected sums thereof. There are now many known examples of higher-dimensional stationary black holes with nontrivial topology, such as the $5$-dimensional ring solutions of \cite{ER} and \cite{PS} which have cross-sectional horizon topology $S^1\times S^2$. However, there are no known examples in which the horizon is a (nontrivial) connected sum of these.

The \emph{near horizon geometry} equations provide another approach to horizon topology in higher dimensions. The idea is to consider, instead of a curvature bound, the precise equations satisfied by the induced degenerate metric on Killing horizons. This has proved useful in the case of extreme (also called degenerate or zero temperature) Killing horizons, see for example \cite{KhuriWoolgar1,KWW}. In \cite{KWW} it is proved, among other things, that for stationary vacuum
extreme
black holes in an $(n+2)$-dimensional spacetime, the fundamental group of the horizon contains an abelian subgroup of finite index which is isomorphic to $\mathbb{Z}^k$ with $k\leq n-2$.
Since extreme horizons constitute a ``set of measure zero'', an obvious question is whether results obtained for zero-temperature black holes using the near horizon geometry equations have some stability when the thermostat is turned up. One purpose of this paper is to generalize the results of \cite{KWW} to nonzero temperature horizons.
Note that each technique listed above deals with a logically different (and in the presence of general time evolution, a physically different) entity: apparent horizons for the technique pioneered by Hawking, event horizon cross-sections for topological censorship, and Killing horizon cross-sections for the near horizon geometries.

Consider an $(n+2)$-dimensional stationary black hole spacetime satisfying the vacuum Einstein equations
\begin{equation}\label{EEE}
R_{\mu\nu}(\mathbf{g})=\frac{2}{n} \Lambda\mathbf{g}_{\mu\nu}.
\end{equation}
According to the rigidity theorem \cite{HI0,HIW,MI}
stationarity generically yields, in addition to an asymptotically timelike Killing field, one or more extra rotational symmetries which altogether produce a Killing field $V$ that is normal to the event horizon. The event horizon is then a Killing horizon, and there exists a \textit{surface gravity} constant $\kappa$ such that
on this surface
\begin{equation}
\pmb{\nabla}_{V}V=\kappa V,
\end{equation}
where $\pmb{\nabla}$ is the Levi-Civita connection for $\mathbf{g}$.
In a neighborhood of each horizon component, Gaussian null coordinates $(u,v,x^i)$ can be introduced so that $V=\partial_{v}$, $u=0$ represents the horizon, $x^i$ are coordinates on the $n$-dimensional compact horizon cross-section $\mathcal{H}$, and $U=\partial_u$ is an outgoing null vector. In these coordinates the spacetime metric then takes the form \cite[Section 3.2]{HI}
\begin{equation}
\mathbf{g} = 2 dv \left(du -u F(u,x) dv -uh_i(u,x) dx^i\right)
 + g_{ij}(u,x) dx^i dx^j.
\end{equation}
Here $g$ is the induced metric on the horizon cross-section and $F(0,x)=\kappa$.
The components of the Ricci tensor in the direction tangent to the cross-section are given in \cite{HI,HIW} by
\begin{equation}\label{ggg}
R_{ij}(\mathbf{g})=R_{ij}(g)-\frac{1}{2}h_i h_j -\nabla_{(i}h_{j)}-\kappa\mathcal{L}_{U}\mathbf{g}_{ij}
-\mathcal{L}_{U}\mathcal{L}_{V}\mathbf{g}_{ij}+O(u),
\end{equation}
where $\mathcal{L}$ denotes Lie differentiation and $\nabla$ is the Levi-Civita connection for $g$. Since $V$ is a Killing field the last term before $O(u)$ vanishes. Thus, with the help of the Einstein equations \eqref{EEE}, taking the limit as $u\rightarrow 0$ produces
\begin{equation}\label{B-E}
R_{ij}(g)-\nabla_{(i}h_{j)}-\frac{1}{2}h_i h_j =\frac{2}{n} \Lambda g_{ij}+2\kappa\chi_{ij}\quad\quad\text{ on }\quad\quad\mathcal{H},
\end{equation}
where $\chi_{ij}=\langle\pmb{\nabla}_{\partial_i}U,\partial_j\rangle$ is the null second fundamental form in the $U$ direction.

Recall that the generalized $m$-Bakry-\'{E}mery Ricci tensor is given by
\begin{equation}\label{bee}
\mathrm{Ric}_{X}^{m}(g)=\mathrm{Ric}(g)
+\frac{1}{2}\mathcal{L}_{X}g-\frac{1}{m}X\otimes X,
\end{equation}
in which $X$ is a 1-form/vector.  Thus, by setting $m=2$ and $X=-h$ equation \eqref{B-E} gives a lower bound for the Bakry-\'{E}mery Ricci curvature of horizon cross-sections
\begin{equation}
\mathrm{Ric}_{-h}^{2}(g)=\frac{2}{n} \Lambda g+2\kappa\chi\quad\quad\text{ on }\quad\quad\mathcal{H}.
\end{equation}
This may then be combined with results concerning Bakry-\'{E}mery Ricci curvature lower bounds to produce restrictions on horizon topology.
In particular, it is typically the case that $\kappa$ is nonnegative as it represents the horizon temperature, so that if in addition $\chi$ is positive semi-definite then the previous results for extreme black holes \cite{KWW} immediately carry over to this realm. However, such semi-definiteness is not a general feature of black hole Killing horizons.  For example, it has been shown by direct computation \cite{ChruHor} that $\chi$ for the Emparan-Reall black ring \cite{ER}, has one negative eigenvalue; it can be inferred from the $m$-Bakry-\'{E}mery splitting theorem obtained in \cite{KWW}, that at least some eigenvalue has to be negative.
With this in mind,
let $\lambda$ denote a lower or upper bound (depending on the sign of $\kappa$) for the eigenvalues of $\chi$, that is
\begin{equation}
\kappa\lambda=\inf_{x\in\mathcal{H}}\min_{ w\in T_{x}\mathcal{H}\atop |w|=1}\kappa\chi(w,w).
\end{equation}
Furthermore let $\mathcal{C}$, $\mathcal{D}$, and $\mathcal{V}$ be constants such that
\begin{equation}\label{,al1}
\mathrm{diam}(\mathcal{H})\leq\mathcal{D},
\quad\quad\quad
\mathrm{Vol}(\mathcal{H})\geq\mathcal{V}, \quad \quad\quad \sup_{\mathcal{H}}\left(|X|+|\nabla\mathrm{div}X|\right)\leq\mathcal{C}.
\end{equation}

\begin{theorem}\label{maincor}
Let $\mathcal{H}$ be a single component compact horizon cross-section in a stationary vacuum spacetime satisfying \eqref{,al1}.

\begin{itemize}
\item [(i)] Assume that $\Lambda\geq 0$. There exists $\kappa_0(n,\lambda,\mathcal{C},\mathcal{D},\mathcal{V})>0$, such that if $|\kappa|\leq\kappa_0$ then $\mathcal{H}$ is not a connected sum $M\# N$, where $M$ and $N$ are compact manifolds having nontrivial fundamental groups, except possibly in the case that $\pi_1(M)=\pi_1(N)=\mathbb{Z}_2$.\smallskip

\item [(ii)] Assume that $\Lambda\geq 0$. There exists $\kappa_0(n,\lambda,\mathcal{C},\mathcal{D})>0$, such that if $|\kappa|\leq\kappa_0$ then the first Betti number satisfies $b_{1}(\mathcal{H})\leq n+2$. Moreover, if $X=df_0$ for some $f_0\in C^{\infty}(\mathcal{H})$ and the assumption $\sup_{\mathcal{H}}\left(|X|+|\nabla\mathrm{div}X|\right)\leq\mathcal{C}$ is replaced by $\sup_{\mathcal{H}} |f_0|\leq \mathcal{C}$, then $b_{1}(\mathcal{H})\leq n$.\smallskip

\item [(iii)] Assume that $\Lambda> 0$. There exists $\kappa_0(n,\lambda,\Lambda)>0$, such that if $|\kappa|\leq\kappa_0$ then
    $\pi_1(\mathcal{H})$ is finite. In particular, de Sitter black rings having horizon cross-sectional topology $S^1\times M$ where $M$ is a compact manifold, do not exist with low temperature.\smallskip

\item [(iv)] Let $\Lambda_0\in\mathbb{R}$. There are only finitely many isomorphism types of $\pi_1(\mathcal{H})$, among horizons satisfying \eqref{,al1} and $\frac{2}{n}\Lambda+\kappa\lambda\geq\Lambda_0$.
\end{itemize}
\end{theorem}

\begin{remark}
It should be noted that the surface gravity $\kappa$ depends on scalings of the Killing field $V$. While there is a canonical normalization in the asymptotically flat setting, in general this is not the case. Thus, it may be desirable in certain situations to restate Theorem \ref{maincor} in terms of the smallness of a quantity invariant under scalings of the Killing field, namely $\kappa\lambda$.
\end{remark}

In the asymptotically flat or asymptotically Kaluza-Klein setting, if there is a $U(1)$ symmetry (this condition is generic \cite{HI0,HIW,MI}) then $\mathbb{RP}^3\# \mathbb{RP}^3$ may be removed from the list of exceptional cases for Theorem \ref{maincor} $(\mathrm{i})$, see \cite[Remark 8]{KWW}. Since the horizon cross-section must be of positive Yamabe type \cite{Galloway,GS}, it follows that low temperature orientable horizons in spacetime dimension 5 can only have the topology of a spherical space, or $S^1\times S^2$. Furthermore, the blackfolds technique \cite{AHO,AO,CER} has been used to infer the existence of new horizons, including new black rings, in asymptotically anti-de Sitter and asymptotically flat spacetimes. The approach also suggests black rings in de Sitter spacetime, but not in the low temperature limit (for bounded horizon area). It would be interesting to determine more precisely the domain of validity of that approach, and of ours.


This theorem may be interpreted as a type of stability for topological restrictions present in the structure of extreme black holes, or rather, low temperature horizons have the same limited topology as zero temperature horizons. The strategy to achieve this result will be to develop an almost splitting theorem in the generalized Bakry-\'{E}mery setting, and then harness the topological conclusions that flow forth. The original almost splitting theorem of Cheeger and Colding \cite{CC}, asserts that if the Ricci curvature is almost nonnegative and there is almost a line, then the manifold almost splits. Thus it is a quantitative form of the Cheeger-Gromoll \cite{CheegerGromoll} splitting theorem, in that it quantifies precisely how far off the manifold is from an exact splitting. From such a quantitative result, topological consequences arise as corollaries, though the consequences are somewhat less restrictive than those implied by an exact splitting. For example, with a diameter upper bound and volume lower bound Yun \cite{Yun}, relying on work of Wei \cite{Wei}, showed that the fundamental group of manifolds with almost nonnegative Ricci curvature is almost abelian, that is, it contains an abelian subgroup of finite index. When the Ricci curvature is nonnegative the splitting theorem leads to knowledge of the structure of this abelian subgroup, namely it is a direct sum of infinite cyclic groups.
An extension of the almost splitting theorem to the Bakry-\'{E}mery setting has been established by Jaramillo \cite{Jaramillo} and Wang and Zhu \cite{WZ}, in the case of a gradient field $X=df$ with $m=\infty$ and $|f|\leq c$; the result in \cite{WZ} requires also a bound on the first derivatives $|\nabla f|\leq c$. Moreover, extensions in this context of the results of Yun and Wei are also given in \cite{Jaramillo}.

Our setting differs from that of \cite{Jaramillo,WZ} in two ways. First, we have a term with negative coefficient $-1/m$ in equation \eqref{bee}, which is not present in the previous works. The sign of this term, however, is beneficial. What makes the current setting more difficult is the second difference, which is that the 1-form $X$ need not be exact.
An almost splitting result in this situation, with $m=\infty$, has been obtained by Zhang and Zhu \cite{ZhangZhu}, in which $X$ is required to be almost zero. For applications to horizons, however, it is necessary to consider the general case where $X$ is neither exact nor small.
It turns out that the advantageous $-1/m$ coefficient is able to compensate for the difficulties arising from large non-gradient $X$, to allow for a version of the almost splitting theorem in this situation. In what follows, the Riemannian and Gromov-Hausdorff distances will be denoted by $d$ and $d_{GH}$, respectively.

\begin{theorem}\label{theoremalmostsplit}
Let $(M,g,X)$ be a complete Riemannian manifold of dimension $n$ with smooth 1-form $X$.  Let $m,r,\epsilon,\mathcal{C}>0$ and $\delta\ge 0$, and assume that $\mathrm{Ric}_X^m(g)\ge -(n-1)\delta g$ together with $\sup_{M}\left(|X|+|\mathrm{div}X|\right)\leq\mathcal{C}$. If $L>2r+1$, and there are points $p,q_{\pm}\in M$ satisfying
\begin{equation}
d(q_-,p)>L,\quad\quad
d(q_+,p)>L,\quad\quad
d(q_-,p)+d(q_+,p)-d(q_-,q_+)<\epsilon,
\end{equation}
then there exists a length space $N$ and a metric ball $B_{r/4}(0,x)\subset \mathbb{R}\times N$ with the product metric, such that
\begin{equation}
d_{GH}\left(B_{r/4}(p),B_{r/4}(0,x)\right)\leq\Upsilon
\end{equation}
where $\Upsilon>0$ may be made arbitrarily small by sending $\epsilon,\delta,L^{-1}\rightarrow 0$.
\end{theorem}

In analogy with the splitting theorem for nonnegative generalized $m$-Bakry-\'{E}mery Ricci curvature \cite{KWW}, the projection of $X$ onto the $\mathbb{R}$-factor and the Bakry-\'{E}mery Ricci curvature in this direction almost vanish in a weak sense described in Theorem \ref{theoremalmostsplit1}.
In the classical case these facts imply that nonnegative Bakry-\'{E}mery Ricci curvature descends to $N$, and it would be of interest to examine to what extent this holds in the current context. An immediate consequence of the almost splitting theorem asserts that the splitting extends to
limit metric spaces under Gromov-Hausdorff convergence, see Corollary \ref{corsequence}. Moreover as described above, the almost splitting theorem
leads to consequences for the fundamental group, in particular we obtain the following characterization.

\begin{theorem}\label{theorem1.2}
Consider a complete Riemannian manifold $(M,g,X)$ of dimension $n$ with smooth 1-form $X$. Let $m>0$, $\delta\ge 0$, and assume that
\begin{equation}
\mathrm{Ric}_X^m(g)\ge -(n-1)\delta g,\quad
\mathrm{diam}(M)\leq\mathcal{D},
\quad
\mathrm{Vol}(M)\geq\mathcal{V}, \quad \sup_{M}\left(|X|+|\nabla\mathrm{div}X|\right)\leq\mathcal{C}.
\end{equation}
There exists $\delta_0\left(n,m,\mathcal{C},\mathcal{D},\mathcal{V}\right)>0$, such that if $\delta\leq\delta_0$ then $\pi_1(M)$ is almost abelian. In particular, such $M$ admit a finite cover whose fundamental group is abelian.
\end{theorem}

Note that the assumptions on $X$ in this theorem and Theorem \ref{maincor}, are stronger than those in Theorem \ref{theoremalmostsplit}.
The proof relies heavily on a volume comparison result, Proposition \ref{volest} below, that plays a role similar to the Bishop-Gromov inequality. Typical volume comparison theorems in the Bakry-\'{E}mery realm use a comparison constant curvature space of higher dimension, which is not sufficient for our purposes since the limit of geodesic ball volume ratios blows up when the dimensions do not coincide. A volume comparison with a model space of the same dimension was achieved by Jaramillo \cite{Jaramillo} in the gradient Bakry-\'{E}mery setting. Surprisingly, establishing such a result in the non-gradient Bakry-\'{E}mery case is quite delicate and requires a new set of ideas. In particular, for topological applications, it is important that the exponential growth of the volume ratio estimate be controlled by the curvature of the model space.
Moreover, our proof requires finite and positive $m$, and it is not clear whether such a result holds when $m=\infty$. If control over the exponential growth of volume ratios is not required, then in the $m=\infty$ case a version of the volume comparison result was given by Zhang and Zhu \cite{ZhangZhu}.

From the volume comparison Proposition \ref{volest}, many classical results for Ricci curvature lower bounds may be extended. As examples of this we obtain generalizations of Anderson's finiteness of fundamental group isomorphism types \cite{Anderson}, and the first Betti number bound of Gallot \cite{Gallot} and Gromov \cite[Theorem 5.21]{Gromov}.

\begin{theorem}\label{theorem1.3''}
Consider a complete Riemannian manifold $(M,g,X)$ of dimension $n$ with smooth 1-form $X$. Let $m>0$, $\delta\ge 0$, and assume that
\begin{equation}\label{tttt}
\mathrm{Ric}_X^m(g)\ge -(n-1)\delta g, \quad\quad\quad\mathrm{diam}(M)\leq\mathcal{D}, \quad\quad\quad \sup_{M}|X|\leq\mathcal{C}.
\end{equation}
\begin{itemize}
\item [(i)] Then there is a function $B\left(\delta,n,m,\mathcal{C},\mathcal{D}\right)$ that yields a bound for the first Betti number and satisfies 
\begin{equation}
b_1(M)\leq B\left(\delta,n,m,\mathcal{C},\mathcal{D}\right),\quad\quad
\lim_{\delta\rightarrow 0}B\left(\delta,n,m,\mathcal{C},\mathcal{D}\right)=n+m.
\end{equation}
More precisely, there is a $\delta_0(n,m,\mathcal{C},\mathcal{D})>0$ such that if $\delta\leq\delta_0$ then $b_1(M)\leq n+m$.
Furthermore, if $X=df_0$ for some $f_0\in C^{\infty}(M)$ and the assumption $\sup_M |X|\leq\mathcal{C}$ is replaced by $\sup_M |f_0|\leq \mathcal{C}$, then the same conclusions hold with $n+m$ replaced by $n$.\smallskip

\item [(ii)] Among the class of manifolds satisfying
\eqref{tttt} together with $\mathrm{Vol}(M)\geq\mathcal{V}$, there are only finitely many isomorphism types of $\pi_1(M)$.
\end{itemize}
\end{theorem}

The paper is organized as follows. In Section \ref{sec2} basic comparison geometry theorems are extended to the current setting, and in Section \ref{sec3} a preliminary quantitative splitting result known as the Abresch-Gromoll excess estimate is established for generalized Bakry-\'{E}mery Ricci curvature. Section \ref{sec4} is dedicated to Hessian estimates for Busemann function stand-ins known as $X$-harmonic replacement functions. With these estimates, together with a segment inequality proven in Section \ref{sec5}, a quantitative version of the Pythagorean Theorem is given in Section \ref{sec6} from which the desired almost splitting theorem follows. Topological consequences of the almost splitting result are established in Section \ref{sec7}, including the proofs of Theorems \ref{maincor}, \ref{theorem1.2}, and \ref{theorem1.3''}. Lastly, in the appendix we derive an extended version of the Cheng-Yau gradient estimate that is appropriate for our purposes.

\subsection*{Acknowledgements}
The authors would like to thank Christina Sormani for discussions that led to the genesis of this paper.

\section{Comparison Geometry}
\label{sec2} \setcounter{equation}{0}
\setcounter{section}{2}

\subsection{Fundamental comparison theory results}

For the Bakry-\'Emery theory with gradient vector field $X=df$, results analogous to those we need are known, or at least are known when $m=\infty$. Typically, results that hold when $m=\infty$ can be established more easily, and with fewer assumptions, in the case of finite positive $m$. The challenge in this paper is to deal with non-gradient vector fields. This is primarily an issue when computing volume integrals, although is not usually a difficulty for comparison estimates that require only line integrals along geodesics.

Along a unit speed geodesic $\gamma:[0,r]\to M$, consider the function
\begin{equation}
\label{eq2.1}
f_{\gamma}(r):=\int_{\gamma} X\cdot {\bf ds}=\int\limits_0^r \langle X(\gamma(s)),\gamma'(s)\rangle ds\ .
\end{equation}
By following the arguments in \cite[Section 2]{WW}, replacing the $f$ of that work with $f_{\gamma}$, and including the helpful term due to the finite $m>0$, a mean curvature/Laplacian comparison result is obtained under the assumption that $f_{\gamma}$ remains bounded. However, we wish to apply these techniques in the setting of noncompact covering spaces, where $f_{\gamma}$ may become unbounded. Nevertheless we are able to overcome these difficulties by exploiting the finite positive parameter $m$. This leaves open the interesting issue of whether the Laplacian comparison theorem of \cite{WW} can be proved for $X$ not gradient, or equivalently Proposition \ref{lemma2.1} below can be proved when $m=\infty$, assuming only a bound on $|X|$.
A related type of Laplacian comparison, with $m=\infty$, has been established in \cite[Proposition 2.1]{ZhangZhu}.

Fix $p\in M$ and let $\rho(x)=d(x,p)$ denote the distance function from $p$. Away from the cut locus this function is smooth and satisfies $|d\rho|=1$ as well as $\nabla_{\nabla \rho}\nabla \rho=0$. The second fundamental form and mean curvature of geodesic spheres are given by
\begin{equation}
A= \hess \rho, \quad\quad\quad
H= \Delta \rho =\tr \hess \rho.
\end{equation}
Recall that this mean curvature satisfies the Riccati equation
\begin{equation}
\label{eq2.4}
H'=\partial_{\rho}H=-|A|^2-\mathrm{Ric}(\nabla \rho, \nabla \rho)=-| \mathring{A} |^2 -\frac{1}{(n-1)}H^2-\mathrm{Ric}(\nabla \rho, \nabla \rho),
\end{equation}
where $\mathring{A}$ is the tracefree part of $A$. The Bakry-\'Emery modified version of the Laplacian, or so called drift Laplacian, is $\Delta_{X}=\Delta-\nabla_X$, and the corresponding  modified mean curvature takes the form
\begin{equation}
\label{eq2.5}
H_X=\Delta_X\rho = H-\nabla_X\rho.
\end{equation}
This yields an augmentation of the Riccati equation
\begin{equation}
\label{eq2.6}
\begin{split}
H_X'=& -|\mathring{A}|^2 -\frac{1}{(n-1)}H_X^2  -\mathrm{Ric}_X^m(\nabla \rho,\nabla \rho)-\frac{2}{(n-1)}H_X\nabla_X\rho-\frac{(n+m-1)}{m(n-1)}\left ( \nabla_X \rho\right )^2\\
=&-|\mathring{A}|^2 -\frac{1}{(n+m-1)}H_X^2 -\mathrm{Ric}_X^m(\nabla \rho,\nabla \rho)\\
&-\frac{1}{(n-1)}\left ( \sqrt{\frac{m}{n+m-1}}H_X + \sqrt{\frac{n+m-1}{m}} \nabla_X \rho \right )^2.
\end{split}
\end{equation}

The comparison spaces will be taken to be the usual simply connected constant curvature models, albeit with a different dimension $d$ depending on the parameter $m$. However, we will not necessarily choose a space of the same dimension. These metrics may be written in geodesic polar coordinates by
\begin{equation}
\label{eq2.7}
\bar{g}_{d,\lambda}= d\rho^2+\ell^2_{\lambda}(\rho)g_{S^{d-1}} ,\quad\quad\quad
\ell_{\lambda}(\rho)=
\begin{cases}
\frac{1}{\sqrt{\lambda}}\sin\sqrt{\lambda}\rho& {\lambda}>0,\\
\rho& {\lambda}=0,\\
\frac{1}{\sqrt{-{\lambda}}}\sinh\sqrt{-{\lambda}}\rho& {\lambda}<0,
\end{cases}
\end{equation}
where $g_{S^{d-1}}$ is the standard round metric on the sphere $S^{d-1}$. Note that
$\ell$ is the solution of the initial value problem
\begin{equation}
\label{eq2.8}
\ell''(\rho)+\lambda \ell(\rho)=0,\quad\quad\quad
\ell(0)=0,\quad
\ell'(0)=1,
\end{equation}
where for convenience we have dropped the subscript $\lambda$. A computation shows that the mean curvature of geodesic spheres in the comparison space is then
\begin{equation}
\label{eq2.9}
{\bar H}_{d}(\rho):= (d-1)\frac{\ell'(\rho)}{\ell(\rho)}=(d-1)
\begin{cases}
\sqrt{\lambda}\cot\sqrt{\lambda}\rho& {\lambda}>0,\\
1/\rho& {\lambda}=0,\\
\sqrt{-{\lambda}}\coth\sqrt{-{\lambda}}\rho& {\lambda}<0.
\end{cases}
\end{equation}

\begin{proposition}[Mean Curvature Comparison]\label{lemma2.1}
Let $m>0$ and $\delta\ge 0$, and assume that $\mathrm{Ric}_X^m(g)\ge -(n-1)\delta g$. Choose a comparison space \eqref{eq2.7} of dimension $d=n+m$ with $\lambda=-\delta$. Then
\begin{equation}
\label{eq2.10}
H_X(\rho)\le {\bar H}_{n+m}(\rho)
\end{equation}
for all $\rho\ge 0$ such that $H_X(\rho)$ is defined. Furthermore, when viewed as an inequality for the drift Laplacian, the corresponding statement holds at all points in the barrier sense.
\end{proposition}

\begin{proof}
Observe that \eqref{eq2.4} and the Bakry-\'Emery Ricci curvature lower bound yield
\begin{align}
\label{eq2.11}
\begin{split}
\left ( \ell^2 H\right )' =&2\ell\ell' H+\ell^2H'\\
\le & 2\ell\ell' H-\frac{1}{(n-1)}\ell^2H^2-\ell^2\mathrm{Ric}(\partial_{\rho},\partial_{\rho})\\
\leq& 2\ell\ell'H -\frac{1}{(n-1)}\ell^2H^2+\frac12\ell^2\pounds_Xg(\partial_{\rho},\partial_{\rho}) -\frac{1}{m}\ell^2\langle X,\partial_{\rho}\rangle^2
+(n-1)\ell^2\delta\\
= & -\left ( \frac{\ell H}{\sqrt{n-1}}-\sqrt{n-1}\ell' \right)^2 +(n-1)\ell'^2+\frac12\ell^2\pounds_Xg(\partial_{\rho},\partial_{\rho})\\
& -\frac{1}{m}\ell^2\langle X,\partial_{\rho}\rangle^2 +(n-1)\ell^2\delta.
\end{split}
\end{align}
On the other hand, \eqref{eq2.8} and \eqref{eq2.9} produce
\begin{align}
\label{eq2.12}
\begin{split}
\left ( \ell^2 {\bar H}_d \right )' =& (d-1)\ell'^2 +(d-1)\delta\ell^2 \\
=& (d-n)\ell'^2+(n-1)\ell'^2 +(d-n)\delta\ell^2 +(n-1)\delta\ell^2.
\end{split}
\end{align}
The two terms on the right-hand side of this equation having coefficient $(n-1)$, appear also in \eqref{eq2.11}. Thus, solving for them in \eqref{eq2.12} and inserting the expression into \eqref{eq2.11} gives
\begin{equation}
\label{eq2.13}
\left ( \ell^2 H\right )' \le \left ( \ell^2 {\bar H}_d \right )'-(d-n)\ell'^2+\ell^2 \nabla_{\partial_{\rho}} \langle X,\partial_{\rho} \rangle -(d-n)\ell^2 \delta -\frac{1}{m}\ell^2\langle X,\partial_{\rho}\rangle^2\ .
\end{equation}
Now integrate this along a radial geodesic $\gamma:[0,r]\to M$, and use $\ell^2 H\rightarrow 0$ as well as $\ell^2 \bar{H}_d\rightarrow 0$ when $\rho\rightarrow 0$, to find
\begin{align}
\label{eq2.14}
\begin{split}
\ell^2(r) H(r) \le &\, \ell^2(r) {\bar H}_d(r) +\ell^2(r)\langle X,\partial_\rho \rangle(r) -\int\limits_0^r 2\ell\ell'\langle X,\partial_{\rho} \rangle d\rho\\
& -\int\limits_0^r \left [  (d-n)\left ( \ell'^2(\rho)+\ell^2(\rho)\delta\right ) +\frac{1}{m}\ell^2(\rho)\langle X,\partial_{\rho}\rangle^2\right ] d\rho.
\end{split}
\end{align}
Using \eqref{eq2.5}, this may be rewritten as
\begin{align}
\label{eq2.15}
\begin{split}
\ell^2(r)\left ( H_X(r)-{\bar H}_d(r)\right )\le & -\int\limits_0^r  \left [(d-n) \ell'^2+ 2\ell\ell'\langle X,\partial_{\rho} \rangle +\frac{\ell^2}{m}\langle X,\partial_{\rho}\rangle^2\right ] d\rho \\
 &-(d-n)\delta \int\limits_0^r \ell^2(\rho)d\rho\\
= & -\int\limits_0^r  \left (\sqrt{m}\ell' +\frac{\ell}{\sqrt{m}}\langle X,\partial_{\rho}\rangle\right )^2 d\rho \\
&\, -(d-n-m)\int\limits_0^r \ell'^2d\rho -(d-n)\delta \int\limits_0^r \ell^2(\rho)d\rho.
\end{split}
\end{align}
Then choosing $d=n+m$ yields the desired result.

Lastly, in order to show that the inequality holds for Laplacians in the barrier sense, one may follow the usual technique of constructing support functions by pushing out slightly along minimizing geodesics.
\end{proof}

\begin{remark}\label{remark2.2}
An estimate for the difference of mean curvatures may be obtained for a comparison space of dimension $n$. In particular, if $|X|\leq \mathcal{C}$ then set $d=n$ in the first and second lines of \eqref{eq2.15}, and use that $\ell \ell'\geq 0$ for $\lambda=-\delta\leq 0$ to find
\begin{equation}
\label{eq2.16}
H_X(r)-{\bar H}_n(r)\le \frac{\mathcal{C}}{\ell^2(r)}\int\limits_0^r 2\ell\ell' d\rho =\mathcal{C}.
\end{equation}
\end{remark}

\subsection{Derived comparison results}



Here various consequences of the mean curvature comparison result will be recorded. Consider a radial function $G_r(\rho)$ which obeys the relations
\begin{equation}
\label{eq2.17}
{\bar \Delta}_d G_r = 1,\quad G_r>0,\quad  G_r'<0 \quad\text{ for }\quad 0<\rho<r ,\quad\quad\quad G_r(r)=G_r'(r)=0,
\end{equation}
where $\bar{\Delta}_d$ denotes the Laplace-Beltrami operator for the constant curvature comparison space of dimension $d$, with $\lambda\leq 0$. Such a function may be easily obtained by integrating the ODE
\begin{equation}
1=\bar{\Delta}_d G_r=\ell^{1-d}\partial_{\rho}\left(\ell^{d-1} \partial_{\rho}G_r\right).
\end{equation}
Observe that this function may also be defined on $M$ using the distance function from $p$. The relations \eqref{eq2.17} then imply the following estimate.

\begin{corollary}\label{corollary2.4}
Under the conditions of Proposition \ref{lemma2.1} with $d=n+m$, we have $\Delta_X G_r\ge 1$.
\end{corollary}

\begin{proof}
Note that \eqref{eq2.10} shows $\Delta_X \rho\leq \bar{\Delta}_{n+m}\rho$.
Therefore, a computation combined with \eqref{eq2.17} produces
\begin{equation}
\label{eq2.18}
\begin{split}
\Delta_X G_r = G_r'(\rho) \Delta_X \rho + G_r'' |d\rho |^2
\ge G_r'(\rho) {\bar \Delta}_{n+m} \rho + G_r'' |d\rho|^2
= {\bar \Delta}_{n+m} G_r= 1.
\end{split}
\end{equation}
\end{proof}

Next, note that the volume forms of $(M,g)$ and the comparison space may be expressed in geodesic polar coordinates by
\begin{equation}\label{677}
dV_g =\mathcal{A} d\rho \wedge dV_{S^{d-1}},\quad\quad\quad
dV_{\bar{g}_{d,\lambda}}=\bar{\mathcal{A}} d\rho\wedge dV_{S^{d-1}},
\end{equation}
where $dV_{S^d}$ is the volume form of the round $d$-sphere and $\bar{\mathcal{A}}=\ell^{d-1}$. According to the first variation of area formula
\begin{equation}
HdV_g =\pounds_{\partial_{\rho}}dV_g =\left(\partial_{\rho}\log\mathcal{A}\right) dV_g,
\end{equation}
so that inequality \eqref{eq2.10} can be rewritten as
\begin{equation}
\partial_{\rho}\log\mathcal{A}-\nabla_{X}\rho\leq\partial_{\rho}\log\bar{\mathcal{A}},
\end{equation}
after setting $d=n+m$. From this it follows that if $|X|\leq \mathcal{C}$ then
\begin{equation}
\partial_{\rho}\left(\frac{e^{-\mathcal{C}\rho}{\mathcal A}}{\bar{\mathcal A}}\right)\leq 0.
\end{equation}
This yields an area comparison result.

\begin{corollary}\label{corollary2.5}
Under the conditions of Proposition \ref{lemma2.1}, if $|X|\le \mathcal{C}$ then $e^{-\mathcal{C}\rho}\bar{\mathcal A}^{-1}{\mathcal A}$ is non-increasing.
\end{corollary}

To finish this section, we derive the analogue of relative volume comparison. Let $\mathrm{Vol}(B_{\rho})$ and $\overline{\mathrm{Vol}}_d(B_{\rho})$ denote the volume of geodesic balls of radius $\rho$ in $(M,g)$ and the comparison space of dimension $d$, respectively. Furthermore, the corresponding volume computed with respect to the weighted volume measure $e^{\mathcal{C}\rho}dV_{\bar{g}_{d,\lambda}}$ will be labeled $\overline{\mathrm{Vol}}_d^{\mathcal{C}}(B_{\rho})$.

\begin{corollary}\label{corollary2.6}
Under the conditions of Proposition \ref{lemma2.1}, if $|X|\le \mathcal{C}$ then for $0<\rho\le r$ we have
\begin{equation}
\label{eq2.20}
\frac{\mathrm{Vol}(B_{\rho})}{\overline{\mathrm{Vol}}_{n+m}^{\mathcal{C}}(B_{\rho})} \ge \frac{\mathrm{Vol}(B_r)}{\overline{\mathrm{Vol}}_{n+m}^{\mathcal{C}}(B_r)}.
\end{equation}
\end{corollary}

\begin{proof}
This follows from the standard argument, see for example \cite[pp. 226--228]{Zhu}.
The only required modifications are to replace $\bar{\mathcal{A}}$ with $e^{C\rho}\bar{\mathcal{A}}$ and use Corollary \ref{corollary2.5}.
\end{proof}

\begin{remark}\label{remark2.7}
Due to the exponential factor, weighted and unweighted volumes are \emph{equivalent} in that for arbitrarily small balls, they are arbitrarily close. Hence, noncollapsing with respect to weighted volume implies noncollapsing in the sense of ordinary Riemannian volume.
\end{remark}

\section{An Abresch-Gromoll Inequality}
\label{sec3} \setcounter{equation}{0}
\setcounter{section}{3}


In the setting of the classical splitting theorem of Cheeger-Gromoll \cite{CheegerGromoll}, a primary hypothesis concerns the existence of a line.
The analogous hypothesis within the context of the almost spitting theorem, is the existence of three points which almost lie on a line, or form a thin triangle. More precisely, let $q_{\pm}\in M$ be points of large distance from one another, with $p\in M$ approximately in the middle and satisfying $L>0$ denotes the separation parameter and $\epsilon>0$ is small then these points should satisfy
\begin{equation}\label{eq3.1}
d(q_-,p)>L,\quad\quad
d(q_+,p)>L,\quad\quad
E(p):=d(q_-,p)+d(q_+,p)-d(q_-,q_+)<\epsilon,
\end{equation}
for a large separation parameter $L>0$ and small parameter $\epsilon>0$.
The quantity $E$ is called the \emph{excess function}, and may be interpreted as the sum of finite scale Busemann functions. More precisely, observe that the counterpart of Busemann functions in the setting of the almost splitting theorem is given by $b_+(x)=d(x,q_+)-d(q_+,p)$ and $b_-(x)=d(x,q_-)-d(q_-,p)$, with
\begin{equation}\label{8io}
b_+(x)+b_-(x)=E(x)-E(p).
\end{equation}
A main step in the Cheeger-Gromoll splitting result is to show that the sum of Busemann functions vanishes. For the almost splitting theorem the corresponding step is to show that the combination \eqref{8io} is small, in fixed balls centered at $p$. Since $E(p)$ is assumed to be small by assumption, this entails proving that $E(x)$ remains small in these domains. This is the content of the Abresch-Gromoll excess estimate \cite{AG} (see also \cite[Theorem 9.1]{Cheeger}). The proof relies heavily on Laplacian/mean curvature comparison, and so uses in a strong way the Bakry-\'Emery-Ricci lower bound as well as the fact that $\delta$ is small. The excess estimate is a tool to be used to obtain control on finite scale Busemann functions or rather their harmonic replacements.

The purpose of this section is to establish an Abresch-Gromoll inequality
for non-gradient Bakry-\'Emery Ricci curvature, without a smallness condition on $X$ (as is assumed in \cite[Theorem 5.3]{ZhangZhu}).
The gradient Bakry-\'Emery version has been proven in \cite[Theorem 2.5]{Jaramillo}, and follows closely the original arguments \cite{AG,Cheeger}. The proof here will proceed along similar lines, making note of appropriate changes required to accommodate the non-gradient field $X$.
Following other authors, we adopt the notation $\Psi\left(\epsilon_1,\dots,\epsilon_k\big\vert c_1,\dots,c_l \right)$ to denote a positive function whose limit vanishes when the first $k$ arguments are simultaneously taken to zero
\begin{equation}
\label{eq3.4}\lim_{\epsilon_1,\dots,\epsilon_k\to 0} \Psi\left(\epsilon_1,\dots,\epsilon_k\big\vert c_1,\dots,c_l \right)=0.
\end{equation}

\begin{theorem}\label{theorem3.1}
Let $m,r>0$ and $\delta\ge 0$, and assume that $\mathrm{Ric}_X^m(g)\ge -(n-1)\delta g$. If $L>2r+1$, and there are points $p,q_{\pm}\in M$ satisfying \eqref{eq3.1}, then
\begin{equation}
\label{eq3.6}
E(x)\le \Psi\left(\epsilon,\delta,L^{-1}\big\vert n,m,r\right)
\end{equation}
for all $x\in B_r(p)$.
\end{theorem}


The proof depends on a result of Abresch and Gromoll \cite{AG} sometimes referred to as a \emph{quantitative maximum principle}. This yields an explicit upper bound in terms of the function $G_r$ of \eqref{eq2.17}, for Lipschitz functions admitting a bound for their drift Laplacian. The arguments of \cite[Theorem 8.12]{Cheeger} and \cite[Proposition 2.4]{Jaramillo} can be adapted straightforwardly. In particular, the proof of \cite[Proposition 2.4]{Jaramillo} goes through, with $\Delta_f$ replaced by $\Delta_X$ and $\Delta_H^{n+4k}$ replaced by ${\bar \Delta}$, so the comparison space is $(n+m)$-dimensional. The condition $|f|\le k$ in that paper has no meaning in our context and can be dropped and not replaced, since no condition of this nature is needed for Laplacian comparison with finite positive $m$.

\begin{lemma}[Quantitative Maximum Principle]\label{proposition3.2}
Let the conditions of Proposition \ref{lemma2.1} hold, and assume that $U:{\overline {B_r(x)}}\to {\mathbb R}$ is a Lipschitz function on a closed ball of radius $r>0$ about $x$ with Lipshitz constant $c$, such that $U(x_0)\leq 0$ for some $x_0\in B_r(x)\setminus \{x\}$ and $U\big\vert_{\partial B_r(x)}\ge 0$. Furthermore, let $b>0$ be such that $\Delta_X U \le b$ in the barrier sense on ${\overline {B_r(x)}}$. Then $U(x)\le bG_r(r_0)+cr_0$ for all $r_0\in (0,d(x,x_0))$.
\end{lemma}

\begin{proof}
Let $y\in B_r(x)$ and $\rho(y)=d(y,x)$. Observe that the function $V(y)=U(y)-bG_r(\rho(y))$ satisfies
\begin{equation}
\Delta_X V\leq 0 \quad\text{ in }\quad B_r(x)\setminus\{x\},\quad\quad V\geq 0\quad\text{ on }\quad\partial B_r(x),\quad\quad V(x_0)<0.
\end{equation}
We claim that for each $r_0\in (0,d(x,x_0))$ there exists a $y_{r_0}\in \partial B_{r_0}(x)$ with $V(y_{r_0})<0$. If not, $V\geq 0$ on $\partial B_{r_0}(x)$, and the maximum principle then implies that $V\geq 0$ on $B_{r}(x)\setminus B_{r_0}(x)$, contradicting $V(x_0)<0$. Using now this claim and the Lipschitz constant $c$, yields the desired conclusion
\begin{equation}
U(x)\leq U(y_{r_0})+c r_0< bG_{r}(r_0)+c r_0.
\end{equation}
\end{proof}

We can now prove Theorem \ref{theorem3.1}. In \cite{Jaramillo} it is stated that, in the gradient Bakry-\'Emery setting, this follows directly from the arguments of \cite[Proposition 9.1]{Cheeger} together with a version of the quantitative maximum principle and Laplacian comparison. This is not completely clear to us, and so we give the proof here in the general nongradient case.

\begin{proof}[Proof of Theorem 3.1]
For $y\in B_{2r+1}(p)$ let $\rho_{\pm}(y)=d(y,q_{\pm})$ and apply Proposition \ref{lemma2.1} to find
\begin{align}
\label{eq3.7}
\begin{split}
\Delta_X E(y) = &\, \Delta_X \rho_{-}(y)+\Delta_X \rho_{+}(y)\\
\le &\, \bar{H}_{n+m}(\rho_{-}(y))+\bar{H}_{n+m}(\rho_{+}(y))\\
\le &\, \Psi\left ( \delta,L^{-1}\big\vert n,m,r\right ).
\end{split}
\end{align}
Let $x\in B_r(p)$, and choose a nonnegative function
$f\in C^{\infty}(\overline{B_{r+1}(x)})$ which vanishes near $\partial B_{r+1}(x)$ and is positive at $p$.
Then the Lipschitz function $U(y)=E(y)-\sqrt{\epsilon} f(y)$ satisfies
\begin{equation}
\Delta_{X}U\leq \Psi+a\sqrt{\epsilon}\quad\text{on}\quad B_{r+1}(x),\quad\quad
U\geq 0\quad\text{on}\quad \partial B_{r+1}(x),\quad\quad U(p)<\epsilon-\sqrt{\epsilon}f(p)<0,
\end{equation}
for some constant $a>0$ and $\epsilon$ sufficiently small, and where \eqref{eq3.1} was used.

Let $r_0\in(0,r)$. For $d(x,p)>r_0$, the quantitative maximum principle may be applied with $x_0=p$, $b=\Psi+a\sqrt{\epsilon}$, and $c=3$, to find
\begin{equation}\label{eq3.8}
E(x)\leq
\sqrt{\epsilon}f(x)+\left(\Psi+a\sqrt{\epsilon}\right) G_{r+1}(r_0)+3r_0.
\end{equation}
On the other hand, for $d(x,p)\leq r_0$, the Lipschitz bound produces
\begin{equation}\label{eq3.800}
E(x)\leq E(p)+3r_0<\epsilon +3r_0.
\end{equation}
Although $G_{r+1}(r_0)\rightarrow\infty$ as $r_0\rightarrow 0$, by choosing $r_0$ small depending on $\epsilon$ and $\Psi$, the right-hand side of \eqref{eq3.8} and \eqref{eq3.800} may be made arbitrarily small by sending $\epsilon,\delta,L^{-1}\rightarrow 0$.
\end{proof}


\section{Hessian Bounds for the $X$-Harmonic Replacement}
\label{sec4} \setcounter{equation}{0}
\setcounter{section}{4}

In analogy with the classical Cheeger-Gromoll splitting theorem, one would like to show that the finite distance Busemann functions $b_{\pm}$ have small Hessians, so that they are `almost linear'. However lack of regularity poses a difficulty, and thus a smooth replacement is used in the form of $X$-harmonic functions. Namely, define the $X$-harmonic replacements by
\begin{equation}
\label{eq4.3}
\Delta_X h_{\pm}=\text{ }\! 0 \quad\text{ on }\quad B_r(p),\quad\quad\quad\quad
h_{\pm}=\text{ }\! b_{\pm}\quad\text{ on }\quad\partial B_r(p).
\end{equation}
Note that $\partial B_{r}(p)$ may not be smooth. In this case, in order to maintain regularity of the solution to the Dirichlet problem, we will approximate $B_{r}(p)$ by a domain with smooth boundary. This may be achieved in various ways. For instance, since $\partial B_{r}(p)$ is compact it may be covered by a finite number of smooth geodesic balls $\{B_j\}$ of arbitrarily small radius. Then $B_{r}(p)\cup_{j} B_j$ approximates $B_{r}(p)$, and has a piecewise smooth boundary that can easily be made $C^{\infty}$ by rounding off the creases.
In light of this discussion, we will proceed assuming that a domain with smooth boundary has been used in place of $B_r(p)$ when necessary, although explicit mention of this will not be made below.

The primary estimates for the smooth replacement function are based on the Laplacian comparison result Proposition \ref{lemma2.1} and the Abresch-Gromoll inequality Theorem \ref{theorem3.1}. In the following statement, integrals with a slash indicate the average value of the integrand over the domain of integration. Hessian estimates were previously established in the case of gradient $X$ in \cite{Jaramillo,WZ}, and for small $X$ in \cite{ZhangZhu}.

\begin{proposition}[Hessian Estimate]\label{proposition4.1}
Assume that the hypotheses of Theorem \ref{theorem3.1} hold, let $|X|+|\mathrm{div}X|\leq\mathcal{C}$ and $\mathrm{Vol}(B_r(p))\geq v>0$, then
\begin{itemize}
\item [(i)] $|h_\pm(x)-b_\pm(x)|\le \Psi\left ( \epsilon,\delta, L^{-1}\big \vert n,m,r\right )$\text{ }\text{ }\text{ }\text{ for }\text{ }\text{ }\text{ }$x\in B_{r}(p)$,
\item [(ii)] $\mathop{\mathlarger{\fint}_{B_r(p)}} \left \vert \nabla h_\pm-\nabla b_\pm\right \vert^2 dV_g\le \Psi\left ( \epsilon,\delta, L^{-1}\big \vert n,m,r,v,\mathcal{C}\right )$,
\item [(iii)]$\mathop{\mathlarger{\fint}_{B_r(p)}} \left \vert \hess h_\pm \right \vert^2 dV_g \le \Psi\left ( \epsilon,\delta, L^{-1}\big \vert n,m,r,v,\mathcal{C}\right)$.
\end{itemize}
\end{proposition}

\begin{remark}
The lower bound $v$ is included on the right-hand side of the estimates $(ii)$ and $(iii)$, even though the volume of $B_r(p)$ is fixed by $r$ for a given manifold $M$, in order to indicate the dependence on quantities relevant to controlling sequences of manifolds.
\end{remark}

\begin{proof}[Proof of (i)]
Observe that according to Proposition \ref{lemma2.1}
\begin{equation}\label{kjh}
\Delta_{X}b_{\pm}\leq \bar{H}_{n+m}(\rho_{\pm})\leq\Psi\left ( \epsilon,\delta, L^{-1}\big \vert n,m,r\right).
\end{equation}
Fix a point $y\in\partial B_{r+1}(p)$, let $\rho(x)=d(x,y)$, and consider the function $G_{2r+1}(\rho(x))$ constructed in \eqref{eq2.17}. By Corollary \ref{corollary2.4} we have $\Delta_{X}G_{2r+1}\geq 1$, so that
\begin{equation}
\Delta_{X}\left(b_{\pm}-h_{\pm}-\Psi G_{2r+1}\right)\leq 0\quad\text{ on }\quad B_{r}(p),
\end{equation}
with
\begin{equation}
b_{\pm}-h_{\pm}-\Psi G_{2r+1}=-\Psi G_{2r+1}\geq -c_1 \Psi\quad\text{ on }\quad \partial B_{r}(p),
\end{equation}
where $c_1>0$ is a constant independent of $\epsilon$, $\delta$, and $L^{-1}$.
The maximum principle then implies that
\begin{equation}
b_{\pm}-h_{\pm}\geq\Psi G_{2r+1}-c_1 \Psi\geq -c_1 \Psi\quad\text{ on }\quad B_{r}(p).
\end{equation}
Next, use the excess estimate Theorem \ref{theorem3.1} to find
$|b_+ +b_-|=|E-E(p)|\leq c_2\Psi$. Since $h_+ +h_-$ is $X$-harmonic and agrees with
$b_+ +b_-$ on $\partial B_{r}(p)$, it follows that $|h_+ + h_-|\leq c_3\Psi$ on $B_{r}(p)$. Finally, on the ball of radius $r$ centered at $p$ it holds that
\begin{equation}
b_{\pm}- h_{\pm}\leq-b_{\mp} -h_{\pm}+c_2 \Psi\leq-(h_{\mp}+h_{\pm})+(c_1+c_2)\Psi
\leq(c_1 +c_2+c_3)\Psi.
\end{equation}
\end{proof}

\begin{proof}[Proof of (ii)]
Observe that since $b_{\pm}$ are Lipschitz functions we have $\Delta b_{\pm}\in H^{-1}(B_{r}(p))$, the dual of the Sobolev space $H^1(B_{r}(p))$. Therefore $\Delta b_{\pm}$ may be paired with an $H^{1}(B_{r}(p))$ function and integrated by parts. It then follows from $(i)$ that
\begin{align}
\label{eq4.4}
\begin{split}
&\int_{B_r(p)}\left \vert \nabla h_{\pm} -\nabla b_{\pm} \right \vert^2 dV_g\\
=& -\int_{B_r(p)}\left ( h_{\pm} - b_{\pm} \right ) \left ( \Delta h_{\pm} - \Delta b_{\pm} \right ) dV_g\\
=& -\int_{B_r(p)}\left ( h_{\pm} - b_{\pm} \right ) \left [ \Delta_X h_{\pm} - \Delta_X b_{\pm} +\nabla_X \left ( h_{\pm}-b_{\pm}\right ) \right ] dV_g\\
\le & \Psi\left ( \epsilon,\delta, L^{-1}\big \vert n,m,r\right) \left [ \int_{B_r(p)}\left \vert \Delta_X b_{\pm} \right \vert dV_g +\mathcal{C} \int_{B_r(p)}\left \vert \nabla h_{\pm} -\nabla b_{\pm} \right \vert dV_g \right ].
\end{split}
\end{align}
Note that $b_{\pm}$ is smooth away from the cut locus of $q_{\pm}$, which is a set of measure zero, and therefore the absolute value $|\Delta_X b_{\pm}|$ is well-defined in the context above. Next, consider the elementary inequality $a-b\leq|a+b|+2a$ for any numbers $a$ and $b$. Using this inequality with $a$ and $b$ representing the integral of the positive and negative parts\footnote{Here the convention for positive and negative parts of a function $f$ is such that $f=f_+ + f_-$ and $|f|=f_+ - f_-$.} of $\Delta_X b_{\pm}$, together with \eqref{kjh} and $|\nabla b_{\pm}|=1$,
yields
\begin{align}
\label{eq4.5}
\begin{split}
\int_{B_r(p)}\left \vert \Delta_X b_{\pm} \right \vert dV_g
\le & \left \vert\int_{B_r(p)}\Delta_X b_{\pm}  dV_g \right \vert +2\left(\sup_{B_r(p)}\Delta_X b_{\pm}\right ) \mathrm{Vol}\left(B_r(p)\right)\\
\le &\mathrm{Vol}\left(\partial B_r(p)\right) +\mathcal{C} \mathrm{Vol}\left( B_r(p)\right)+2\Psi \mathrm{Vol}\left(B_r(p)\right).
\end{split}
\end{align}
Moreover by Young's inequality
\begin{equation}
\label{eq4.6}
\int_{B_r(p)}\left \vert \nabla h_{\pm} -\nabla b_{\pm} \right \vert dV_g\le  \frac{1}{2}\int_{B_r(p)}\left \vert \nabla h_{\pm} -\nabla b_{\pm} \right \vert^2 dV_g +\frac{1}{2}\mathrm{Vol}\left( B_r(p)\right) .
\end{equation}
If $\Psi$ is sufficiently small while $\mathcal{C}$ is held fixed, we then have
\begin{equation}
\label{eq4.7}
\mathop{\mathlarger{\fint}_{B_r(p)}}\left \vert \nabla h_{\pm} -\nabla b_{\pm} \right \vert^2 dV_g
\le 2\Psi\left ( \epsilon,\delta, L^{-1}\big \vert n,m,r,\mathcal{C}\right)\left ( \frac{\mathrm{Vol}\left( \partial B_r(p)\right)}{\mathrm{Vol}\left( B_r(p)\right)} +2\mathcal{C}+2\right ) .
\end{equation}
Lastly, observe that Corollary \ref{corollary2.5} produces $\mathrm{Vol}(\partial B_{r}(p))\leq e^{\mathcal{C}r}\overline{\mathrm{Vol}}_{n+m}(\partial B_r)$,
and by assumption $\mathrm{Vol}(B_r(p))\geq v$.
\end{proof}

\begin{proof}[Proof of (iii)] From \cite[Lemma 4]{KWW} we have, for functions $u\in C^{\infty}(M)$, the Bochner formula
\begin{equation}
\label{eq4.8}
\Delta_X \left ( \left \vert \nabla u \right \vert^2 \right ) = 2\left \vert \hess u \right \vert ^2 +2\nabla_{\nabla u} \left ( \Delta_X u\right ) +2\mathrm{Ric}_X^m (\nabla u ,\nabla u) +\frac{2}{m}\left ( X(u)\right )^2\ .
\end{equation}
Setting $u=h_{\pm}$ and recalling that these functions are $X$-harmonic, as well as the fact that $\mathrm{Ric}_X^m\ge -(n-1)\delta g$, gives rise to
\begin{align}
\label{eq4.9}
\begin{split}
\left \vert \hess h_{\pm} \right \vert ^2 \le&  \frac12 \Delta_X \left ( \left \vert \nabla h_{\pm} \right \vert^2 \right ) +(n-1)\delta\left \vert \nabla h_{\pm}\right \vert^2 -\frac{1}{m}\left ( \nabla_X h_{\pm}\right )^2 \\
=& \frac12 \Delta_X \left ( \left \vert \nabla h_{\pm} \right \vert^2 -
1\right ) +(n-1)\delta\left \vert \nabla h_{\pm}\right \vert^2 -\frac{1}{m}\left ( \nabla_X h_{\pm}\right )^2.
\end{split}
\end{align}
Now introduce a nonnegative cut-off function $\phi \in C^{\infty}_c(B_r(p))$ with $\phi\equiv 1$ on $B_{r/2}(p)$, and $|\Delta_X \phi |+|\nabla \phi |\le c( n,m,r,\mathcal{C})$. The construction of this cut-off function will be addressed below. Then multiplying \eqref{eq4.9} by $\phi$ and integrating by parts yields
\begin{align}
\label{eq4.11}
\begin{split}
\int_{B_{r/2}(p)} \left \vert \hess h_{\pm} \right\vert^2 dV_g\leq&
\int_{B_{r}(p)} \phi \left \vert \hess h_{\pm} \right\vert^2 dV_g\\
\leq &\frac{1}{2}\int_{B_r(p)} \left (\Delta_{X} \phi +2\nabla_X \phi +\phi \mathrm{div} X\right ) \left ( \left \vert \nabla h_{\pm} \right \vert^2-\left \vert \nabla b_{\pm} \right \vert^2 \right ) dV_g\\
& + \int_{B_r(p)}\phi\left[(n-1)\delta\left \vert \nabla h_{\pm} \right \vert^2 -\frac{1}{m}\left ( \nabla_X h_{\pm} \right )^2\right] dV_g,
\end{split}
\end{align}
where we have used that $|\nabla b_{\pm}|^2=1$ a.e.
The last term on the right-hand side has an advantageous sign, while the others may be estimated by part $(ii)$. Together with the volume comparison of Corollary \ref{corollary2.6}, which implies that
\begin{equation}
\frac{\mathrm{Vol}(B_r(p))}{\mathrm{Vol}(B_{r/2}(p))}\leq
\frac{\overline{\mathrm{Vol}}_{n+m}^{\mathcal{C}}(B_r)}
{\overline{\mathrm{Vol}}_{n+m}^{\mathcal{C}}(B_{r/2})},
\end{equation}
the desired result is achieved.

Finally, we consider the existence of a cut-off function $\phi$ with the necessary properties. This is shown in \cite[Theorem 6.33]{CC} in the setting of Ricci lower bounds. That proof goes through here, \emph{mutatis mutandis}, modulo the use of the Cheng-Yau gradient estimate \cite{ChengYau}. Specifically, the proof proceeds by constructing exact solutions of ordinary differential equations on the comparison space of dimension $n$, which for our case becomes the comparison space of dimension $n+m$. Laplacian comparison then yields differential inequalities which, for us, hold for the drift Laplacian $\Delta_X$ on an $n$-manifold, as in Corollaries \ref{corollary2.4} and \ref{corollary2.5}. One obtains \emph{en lieu} of \cite[Equation 6.59]{CC} the differential equation $\Delta_X \phi = \psi'' |\nabla k|^2 +\psi'\delta$ for $\phi$, with $k$ and $\psi$ as defined in \cite{CC}. As per that reference, the construction is then complete, and the desired properties then follow from the Cheng-Yau estimate. This gradient estimate, which is applied to $\Delta_{X}k=\delta$ in our setting,
requires modification that is provided in the Appendix.
\end{proof}

\section{The Segment Inequality}
\label{sec5} \setcounter{equation}{0}
\setcounter{section}{5}

The Cheeger-Gromoll splitting theorem is established by first finding pointwise estimates for the Laplacian of Busemann functions and then for their Hessians.
In the context of the Cheeger-Colding almost splitting theorem, one is only able to
find estimates for volume integrals of these quantities over certain regions.
These in turn may be brought closer to the pointwise realm, by showing that they imply integral estimates for these quantities along geodesics.
The primary tool used to achieve this goal is the \textit{segment inequality}. Generalizations of the original inequality from \cite{CC} have been obtained in the gradient Bakry-\'{E}mery setting in \cite{Jaramillo,WZ}, and a version for the non-gradient case appears in \cite{ZhangZhu}. Here, the short proof is included for completeness in the general case following the renditions of \cite{Gallot,Richard}.

\begin{lemma}\label{proposition5.1}
Let $m,r>0$ and $\delta\geq 0$, and assume that $\mathrm{Ric}_{X}^m(g)\geq -(n-1)\delta g$ in addition to $|X|\leq\mathcal{C}$.
Let $f:B_{2r}(p)\rightarrow\mathbb{R}_{\geq 0}$, consider domains $\Omega_1,\Omega_2\subset B_{r}(p)$ with $x_1\in\Omega_1$ and $x_2\in\Omega_2$, and define
\begin{equation}
\label{eq5.1}
{\mathcal F}_{f} (x_1,x_2) = \sup_{\gamma} \int_0^{d(x_1,x_2)} f\circ\gamma(s) ds ,
\end{equation}
where the supremum is taken over minimizing unit speed geodesics $\gamma$ joining $x_1$ to $x_2$. Then there exists a constant $c$ depending on $n$, $m$, $r$, $\delta$, and $\mathcal{C}$ such that
\begin{equation}
\label{eq5.2}
\int_{\Omega_1\times \Omega_2} {\mathcal F}_{f} dV_g \wedge dV_g \le
c\left (\mathrm{Vol}(\Omega_1)
+\mathrm{Vol}(\Omega_2)\right )
\int_{B_{2r}(p)} f dV_g.
\end{equation}
\end{lemma}

\begin{proof}
Up to a set of measure zero, each pair of points $(x,y)\in \Omega_1 \times\Omega_2$
is joined by a unique minimal geodesic $\gamma_{xy}:[0,d(x,y)]\to B_{2r}(p)$. Thus, for the current purpose, the integral of \eqref{eq5.1} may be evaluated along such geodesics. Now write
\begin{equation}
\label{eq5.3}
{\mathcal F}_{f}^+(x,y)= \int_{d(x,y)/2}^{d(x,y)} f\circ\gamma_{xy}(s)ds,\quad\quad\quad
{\mathcal F}_{f}^-(x,y)= \int_0^{d(x,y)/2} f\circ\gamma_{xy}(s)ds,
\end{equation}
so that
\begin{equation}
{\mathcal F}_{f}(x,y)={\mathcal F}_{f}^+(x,y)+ {\mathcal F}_{f}^-(x,y).
\end{equation}
Fix $x\in \Omega_1$ and study ${\mathcal F}_{f}^+(x,\cdot)$ by writing its integrand using geodesic polar coordinates about $x$. Since the cut-locus of $x$ is a set of measure zero, we have
\begin{equation}
\label{eq5.4}
\int_{\Omega_2} {\mathcal F}_{f}^+(x,\cdot)e^{-\mathcal{C}\rho}dV_g = \int_{S^{n-1}}\int_{I_{\theta}} {\mathcal F}_{f}^+(x,\exp_x(\rho\theta)){\mathcal A}_x e^{-\mathcal{C}\rho}d\rho dV_{S^{n-1}},
\end{equation}
where $dV_g=\mathcal{A}_x d\rho\wedge dV_{S^{n-1}}$ is the volume form expressed in polar coordinates centered at $x$, and $I_{\theta}=\{ \rho \big\vert \exp_x(\rho\theta)\in \Omega_2\}$.
By the area comparison Corollary \ref{corollary2.5} we have
\begin{equation}
\label{eq5.5}
\frac{e^{-\mathcal{C}t}{\mathcal A}_x(t)}{\bar{\mathcal A}(t)}\ge \frac{e^{-\mathcal{C}\rho}{\mathcal A}_x(\rho)}{\bar{\mathcal A}(\rho)}\quad\text{ for }\quad
t\leq \rho,
\end{equation}
and therefore
\begin{align}
\label{eq5.6}
\begin{split}
{\mathcal F}_{f}^+(x,\exp_x(\rho\theta)){\mathcal A}_x(\rho)e^{-\mathcal{C}\rho} =& \left ( \int_{\rho/2}^{\rho} f(\exp_x(t\theta))dt \right ) {\mathcal A}_x(\rho)e^{-\mathcal{C}\rho}\\
\le &  c_1\int_{\rho/2}^{\rho}f(\exp_x(t\theta)) {\mathcal A}_x(t) e^{-\mathcal{C}t} dt,
\end{split}
\end{align}
where $c_1$ depends on $n$, $m$, $r$, $\delta$, and $\mathcal{C}$.
It follows that
\begin{align}
\label{eq5.7}
\begin{split}
e^{-2r\mathcal{C}}\int_{\Omega_2}{\mathcal F}_{f}^+(x,\cdot) dV_g\leq&
\int_{\Omega_2}{\mathcal F}_{f}^+(x,\cdot)e^{-\mathcal{C}\rho} dV_g \\
\le & c_1\int_{S^{n-1}} \int_{I_{\theta}} \int_{\rho/2}^\rho f(\exp_x(t\theta)){\mathcal A}_x(t) dt d\rho dV_{S^{n-1}} \\
\le & c_1\int_0^{2r}\left ( \int_{S^{n-1}}  \int_0^r f(\exp_x(t\theta)){\mathcal A}_x(t) dt dV_{S^{n-1}}\right) d\rho \\
\le & 2rc_1\int_{B_{2r}(p)}f dV_g.
\end{split}
\end{align}
Integrate once more, over $\Omega_1$, to obtain the desired estimate for $\int_{\Omega_1\times \Omega_2}{\mathcal F}_{f}^+ dV_g \wedge dV_g$. Lastly, the ${\mathcal F}_{f}^-$ contribution is dealt with by interchanging the roles of $\Omega_1$ and $\Omega_2$, and repeating the argument above.
\end{proof}

To see how this leads to `almost pointwise' bounds, choose now $\Omega_1=B_{r}(x_1)\subset B_{2r}(p)$ and $\Omega_2=B_{r}(x_2)\subset B_{2r}(p)$. If $\mathcal{F}_f$ were continuous, then the mean value theorem for integrals implies that there are points $x_1^*\in B_r(x_1)$, $x_2^*\in B_r(x_2)$ such that the left-hand side of \eqref{eq5.2} can be replaced by ${\mathcal F}_{f} (x_1^*,x_2^*)\mathrm{Vol}(B_r(x_1))\mathrm{Vol}(B_r(x_2))$, and so we obtain
\begin{equation}
\label{eq5.8}
{\mathcal F}_{f} (x_1^*,x_2^*)\le c\left(\frac{1}{\mathrm{Vol}(B_r(x_1))}
 +\frac{1}{\mathrm{Vol}(B_r(x_2))}\right) \int_{B_{4r}(p)}f dV_g.
\end{equation}
In general $\mathcal{F}_f$ may not be continuous, in which case the Markov inequality may be employed in place of the mean value theorem in order to achieve the same estimate with a different constant $c$. Recall that the Markov inequality states that for $\eta>0$ and a nonnegative measurable function $u$ on a domain $\Omega$, it holds that
\begin{equation}
\frac{1}{\eta}\int_{\Omega}u\geq |\{x\in\Omega\mid u(x)\geq\eta\}|.
\end{equation}
Note that the bound \eqref{eq5.8} becomes useless when taking Gromov-Hausdorff limits, if the volume of the ball $B_{2r}(p)$ approaches zero, that is, if \emph{collapsing} occurs. However, with the volume comparison result Corollary
\ref{corollary2.6}, collapse of a ball of fixed radius can be avoided with the assumption of a total volume lower bound $\mathrm{Vol}(M)\geq\mathcal{V}>0$ and diameter upper bound $\mathrm{diam}(M)\leq\mathcal{D}$, as in Theorem \ref{theorem1.2}. The following estimates are the main application of the segment inequality to be used in the almost splitting result.

\begin{prop}\label{lemma5.2}
Assume that the hypotheses of Theorem \ref{theorem3.1} hold, together with $|X|+|\mathrm{div}X|\leq\mathcal{C}$ and $\mathrm{Vol}(B_r(p))\geq v>0$, and let
$\Psi=\Psi\left ( \epsilon,\delta, L^{-1}\big \vert n,m,r,v,\mathcal{C}\right )$
be as in Proposition \ref{proposition4.1}.
Let $x,y,z\in B_{r/4}(p)$ be such that $x$ and $y$ lie on a level set of $h_+$, and $z$ lies on a minimizing geodesic connecting $q_+$ to $y$.
Then there exist $x_* \in B_{\varrho}(x)$, $y_*\in B_{\varrho}(y)$, and $z_*\in B_{\varrho}(z)$ satisfying the following properties, where $\varrho=\Psi^{3\varsigma}$ with $\varsigma=\frac{1}{45(n+m)}$. There is a minimal geodesic $\sigma(s)$ from $z_*$ to $y_*$, such that for almost every $s\in[0,d(y_*,z_*)]$ a unique minimal geodesic $\tau_s$ connects $x_*$ to $\sigma(s)$, and
\begin{itemize}
\item[(i)] $\int_0^{d(y_*,z_*)}\left \vert \nabla h_+(\sigma(s)) -\sigma'(s) \right \vert ds\le \Psi^{\varsigma}$,
\item[(ii)] $\int_0^{d(y_*,z_*)} \int_0^{d(x_*,\sigma(s))} \left \vert \hess h_+(\tau_s(t)) \right \vert dt ds\le \Psi^{\varsigma}$.
\end{itemize}
A similar statement holds for $h_-$.
\end{prop}

\begin{proof}
We will use the following notation for volumes $|\Omega|=\mathrm{Vol}(\Omega)$. Let $\varrho$ be small enough so that $B_{\varrho}(x)\subset B_{r/2}(p)$. Observe then that the segment inequality Lemma \ref{proposition5.1} with $f=|\hess h_+|$, together with the Markov inequality, yields the existence of $x_*\in B_{\varrho}(x)$ such that
\begin{align}\label{al}
\begin{split}
|B_{\varrho}(x)|\int_{B_{r/2}(p)}\mathcal{F}_{|\hess h_+|}(x_*,\cdot)dV_g\leq&
c_1\int_{B_{\varrho}(x)\times B_{r/2}(p)}\mathcal{F}_{|\hess h_+|} dV_g\wedge dV_g\\
\leq &c_2\left(|B_{\varrho}(x)|+|B_{r/2}(p)|\right)\int_{B_r(p)}|\hess h_+| dV_g.
\end{split}
\end{align}
On the other hand, if $B_{\varrho}(y),B_{\varrho}(z)\subset B_{r/4}(p)$ then the segment inequality with $f=\mathcal{F}_{|\hess h_+|(x_*,\cdot)}$ and separately $f=|\nabla h_+ -\nabla b_+|$, again combined with Markov's inequality, provides $y_*\in B_{\varrho}(y)$ and $z_*\in B_{\varrho}(z)$ so that
\begin{align}\label{alp}
\begin{split}
&|B_{\varrho}(y)||B_{\varrho}(z)|\left(\mathcal{F}_{|\nabla h_+ -\nabla b_+|}(y_*,z_*)+\mathcal{F}_{\mathcal{F}_{|\hess h_+|}(x_*,\cdot)}(y_*, z_*)\right)\\
\leq&
c_3\int_{B_{\varrho}(y)\times B_{\varrho}(z)}\left(\mathcal{F}_{|\nabla h_+ -\nabla b_+|}+\mathcal{F}_{\mathcal{F}_{|\hess h_+|}(x_*,\cdot)}\right) dV_g\wedge dV_g\\
\leq &c_4\left(|B_{\varrho}(y)|+|B_{\varrho}(z)|\right)\left(\int_{B_{r/2}(p)}|\nabla h_+ -\nabla b_+| dV_g+
\int_{B_{r/2}(p)}\mathcal{F}_{|\hess h_+|(x_*,\cdot)} dV_g\right).
\end{split}
\end{align}
Inequalities \eqref{al} and \eqref{alp} then give
\begin{align}
\begin{split}
&\mathcal{F}_{|\nabla h_+ -\nabla b_+|}(y_*,z_*)+\mathcal{F}_{\mathcal{F}_{|\hess h_+|}(x_*,\cdot)}(y_*, z_*)\\
\leq &c_4\left(\frac{1}{|B_{\varrho}(y)|}+\frac{1}{|B_{\varrho}(z)|}\right)
\int_{B_r(p)}|\nabla h_+ -\nabla b_{+}|dV_g
\\
&+\frac{c_2 c_4 \left(|B_{\varrho}(y)|+|B_{\varrho}(z)|\right)}{|B_{\varrho}(y)||B_{\varrho}(z)|}
\frac{\left(|B_{\varrho}(x)|+|B_{r/2}(p)|\right)}{|B_{\varrho}(x)|}
\int_{B_r(p)}|\hess h_+| dV_g\\
\leq & \frac{c_5 |B_{r}(p)|^3}{|B_{\varrho}(x)||B_{\varrho}(y)||B_{\varrho}(z)|}
\Psi^{1/2},
\end{split}
\end{align}
where in the last line Proposition \ref{proposition4.1} was used along with H\"{o}lder's inequality. The volume comparison result, Corollary \ref{corollary2.6}, implies that
\begin{equation}
|B_{\varrho}(x)|\geq\frac{\overline{\mathrm{Vol}}_{n+m}^{\mathcal{C}}(B_{\varrho})}
{\overline{\mathrm{Vol}}_{n+m}^{\mathcal{C}}(B_{2r})}|B_{2r}(x)|
\geq c_6 \varrho^{n+m}|B_{r}(p)|,
\end{equation}
for some constant $c_6>0$ depending only on $n$, $m$, $r$, and $\mathcal{C}$. The same estimate holds for $|B_{\varrho}(y)|$ and $|B_{\varrho}(z)|$. Using $\varrho=\Psi^{\frac{1}{15(n+m)}}$, it then follows that
\begin{equation}\label{167}
\mathcal{F}_{|\nabla h_+ -\nabla b_+|}(y_*,z_*)+\mathcal{F}_{\mathcal{F}_{|\hess h_+|}(x_*,\cdot)}(y_*, z_*)
\leq c_7 \varrho^{-3(n+m)}\Psi^{1/2}\leq \Psi^{1/4}
\end{equation}
for $\Psi$ sufficiently small. Note that this inequality for the second term on the left-hand side, implies statement $(ii)$.

In order to obtain statement $(i)$, observe that \eqref{167} yields
\begin{equation}
\int_{0}^{d(y_*,z_*)}|\nabla h_+ -\nabla b_+|(\sigma(s))ds\leq \Psi^{1/4}.
\end{equation}
Furthermore
\begin{equation}
|\nabla b_+(\sigma(s))-\sigma'(s)|^2 =2-2\langle \sigma'(s),\nabla b_+(\sigma(s))\rangle=2\left(1-\partial_s b_{+}(\sigma(s))\right),
\end{equation}
so that
\begin{align}
\begin{split}
\frac{1}{2}\int_{0}^{d(y_*,z_*)}|\nabla b_+(\sigma(s))-\sigma'(s)|^2 ds
=& d(y_*,z_*)+b_{+}(z_*)-b_{+}(y_*)\\
\leq& d(y,z)+d(z,q_+)-d(y,q_+)+c_8\varrho\\
=&c_8 \Psi^{3\varsigma}
\end{split}
\end{align}
where in the last step, the hypothesis that $z$ lies on a minimizing geodesic connecting $q_+$ to $y$, was used. Applying H\"{o}lder's inequality once more then produces
\begin{align}
\begin{split}
\int_{0}^{d(y_*,z_*)}|\nabla h_+(\sigma(s)) -\sigma'(s)|ds\leq&
\int_{0}^{d(y_*,z_*)}|\nabla h_+ -\nabla b_+|(\sigma(s))ds\\
&+\int_{0}^{d(y_*,z_*)}|\nabla b_+(\sigma(s)) -\sigma'(s)|ds\\
\leq & c_9 \Psi^{3\varsigma/2},
\end{split}
\end{align}
which yields the desired result for $\Psi$ sufficiently small.
\end{proof}

\section{The Almost Splitting of a Metric Ball}
\label{sec6} \setcounter{equation}{0}
\setcounter{section}{6}

We have now collected all the estimates that enter into the proof of the quantitative Pythagorean theorem. This result, Lemma \ref{lemma6.1} below, is the last step needed before proceeding to the almost splitting theorem. Quasi-right geodesic triangles are constructed based on level sets of $h_+$, and are shown to almost satisfy the Pythagorean relation. The statement provided here is slightly different than those of \cite[Proposition 2.8]{Jaramillo} and \cite[Lemma 3.2]{WZ}, and for this reason we include the proof which follows along similar lines to those of \cite[Lemma 9.16]{Cheeger}.

\begin{lemma}\label{lemma6.1}
Assume that the hypotheses of Proposition \ref{lemma5.2} hold.
Let $x,y,z\in B_{r/4}(p)$ be such that $x$ and $y$ lie on a level set of $h_+$, and $z$ lies on a minimizing geodesic connecting $q_+$ to $y$, then for $\Psi$ sufficiently small
\begin{equation}
d(x,y)^2 +d(y,z)^2 -d(x,z)^2 \leq \Psi^{\varsigma/2}.
\end{equation}
\end{lemma}

\begin{proof}
Let $x_* \in B_{\varrho}(x)$, $y_*\in B_{\varrho}(y)$, and $z_*\in B_{\varrho}(z)$ be as in Proposition \ref{lemma5.2}, and let $\sigma(s)$ be a minimal geodesic from $z_*$ to $y_*$, such that for almost every $s\in[0,T=d(y_*,z_*)]$ a unique minimal geodesic $\tau_s$ connects $x_*$ to $\sigma(s)$. Let $l(s)$ denote the length of $\tau_s$. Note that $l'(s)$ exists for almost all s, and by the first variation of arclength $l'(s) = \left \langle\sigma'(s), \tau_s'(l(s))\right \rangle$.  We then find that
\begin{align}\label{eq6.1}
\begin{split}
\frac{1}{2}\left(d(x_*,y_*)^2-d(x_*,z_*)^2\right) =&\frac{1}{2}\left(l(T)^2-l(0)^2\right)\\
=& \int_0^T l(s)l'(s) ds\\
=&\int_0^T l(s) \left \langle\sigma'(s), \tau_s'(l(s))\right \rangle ds\\
\leq& \int_0^T l(s) \left \langle \nabla h_+(\sigma(s)), \tau_s'(l(s)) \right \rangle ds +r \Psi^{\varsigma}\\
=&\int_0^T \int_{0}^{l(s)} \left \langle \nabla h_+(\tau_s(l(s))), \tau_s'(l(s)) \right \rangle dt ds +r \Psi^{\varsigma},
\end{split}
\end{align}
where in the second to last line statement $(i)$ of Proposition \ref{lemma5.2} was used. Next observe that
\begin{equation}
\left \langle \nabla h_+(\tau_s(l(s))), \tau_s'(l(s)) \right \rangle
=\left \langle \nabla h_+(\tau_s(t)), \tau_s'(t) \right \rangle
+\int_t^{l(s)}\hess h_+(\tau_{s}'(\bar{t}),\tau_{s}'(\bar{t}))d\bar{t},
\end{equation}
and by Proposition \ref{lemma5.2} $(ii)$
\begin{equation}
\int_{0}^{T}\int_t^{l(s)}\left|\hess h_+(\tau_{s}'(\bar{t}),\tau_{s}'(\bar{t}))\right|d\bar{t}ds\leq\Psi^{\varsigma}.
\end{equation}
It follows that
\begin{align}\label{jhui}
\begin{split}
\frac{1}{2}\left(d(x_*,y_*)^2-d(x_*,z_*)^2\right) \leq&
\int_0^T \int_{0}^{l(s)} \left \langle \nabla h_+(\tau_s(t)), \tau_s'(t) \right \rangle dt ds +2r \Psi^{\varsigma}\\
=&\int_0^T \left(h_+(\sigma(s))-h_+(x_*)\right)ds+2r \Psi^{\varsigma}.
\end{split}
\end{align}
According to Proposition \ref{proposition4.1} $(i)$
\begin{equation}
|h_+(x_*)-h_+(x)|\leq|h_+(x_*)-b_+(x_*)|+|b_+(x_*)-b_+(x)|
+|b_+(x)-h_+(x)|\leq 2\Psi+\varrho\leq 4\Psi^{3\varsigma},
\end{equation}
and a similar estimate holds for $|h_+(y_*)-h_+(y)|$.
Therefore since $x$ and $y$ lie on the same level set of $h_+$, we find that
\begin{align}
\begin{split}
h_+(\sigma(s))-h_+(x_*)=& h_+(\sigma(s))-h_+(y_*)+h_+(y_*)-h_+(x_*)\\
\leq& h_+(\sigma(s))-h_+(\sigma(T))+h_+(y)-h_+(x)+6\Psi^{3\varsigma}\\
=&-\int_s^T \partial_{\bar{s}}h_{+}(\sigma(\bar{s}))d\bar{s}+6\Psi^{3\varsigma}\\
=&\int_s^T\left[\langle\sigma'(\bar{s}),\sigma'(\bar{s})-\nabla h_+(\sigma(\bar{s}))\rangle-1\right]d\bar{s}+6\Psi^{3\varsigma}\\
&\leq s-T +r\Psi^{\varsigma}+6\Psi^{3\varsigma},
\end{split}
\end{align}
where in the last line Proposition \ref{lemma5.2} $(i)$ was used. Combining this with \eqref{jhui} produces
\begin{equation}
\frac{1}{2}\left(d(x_*,y_*)^2-d(x_*,z_*)^2\right)\leq -\frac{1}{2}T^2 +4r\Psi^{\varsigma}+6\Psi^{3\varsigma}=-\frac{1}{2}d(y_*,z_*)^2 +(4r+1)\Psi^{\varsigma},
\end{equation}
from which the desired result is obtained.
\end{proof}

We are now in a position to establish the nongradient Bakry-\'{E}mery almost splitting theorem. In the classical setting of nonnegative Bakry-\'{E}mery Ricci curvature \cite{KWW}, in addition to the metric splitting $M=\mathbb{R}\times N$, the projection of $X$ onto the linear $\mathbb{R}$-factor as well as the Bakry-\'{E}mery Ricci curvature in this direction, both vanish. A weak version of this conclusion holds in the setting of almost rigidity, in the form of \eqref{poiu} below. In the classical case these facts imply that nonnegative Bakry-\'{E}mery Ricci curvature descends to $N$, and it would be of interest to examine the extent to which this holds in the current context.

\begin{theorem}\label{theoremalmostsplit1}
Let $(M,g,X)$ be a complete Riemannian manifold of dimension $n$ with smooth vector field $X$.  Let $m,r,\epsilon,\mathcal{C}>0$ and $\delta\ge 0$, and assume that $\mathrm{Ric}_X^m(g)\ge -(n-1)\delta g$ together with $\sup_{M}\left(|X|+|\mathrm{div}X|\right)\leq\mathcal{C}$. If $L>2r+1$, and there are points $p,q_{\pm}\in M$ satisfying
\begin{equation}
\dist(q_-,p)>L,\quad\quad
\dist(q_+,p)>L,\quad\quad
\dist(q_-,p)+\dist(q_+,p)-\dist(q_-,q_+)<\epsilon,
\end{equation}
then there exists a length space $N$ and a metric ball $B_{r/4}(0,x)\subset \mathbb{R}\times N$ with the product metric, such that
\begin{equation}\label{1567}
d_{GH}\left(B_{r/4}(p),B_{r/4}(0,\mathrm{n})\right)\leq \Psi^{\varsigma/5}.
\end{equation}
Here $N$ is the level set $h_+^{-1}(0)$ endowed with the subspace metric arising from $M$. Moreover, the projection of $X$ onto the $\mathbb{R}$-factor and the Bakry-\'{E}mery Ricci curvature in this direction almost vanish in the following sense
\begin{equation}\label{poiu}
\int_{B_r(p)}\left[\langle \nabla h_+,X\rangle^2
+\left(\mathrm{Ric}_X^m(g)+(n-1)\delta g\right)(\nabla h_+,\nabla h_+) \right]dV_g\leq \Psi.
\end{equation}
\end{theorem}

\begin{proof}
A Gromov-Hausdorff approximation or $\Psi^{\varsigma/4}$-isometry $\Xi:B_{r/4}(p)\rightarrow B_{r/4}(0,\mathrm{n})$ may be constructed by $\Xi(x)=(h_+(x),\hat{x})$, where $\hat{x}$ minimizes the distance to $x$ among points of $N$. In particular, with the help of Lemma \ref{lemma6.1} it can be shown \cite[Proposition 3.6]{WZ} that for $x,y\in B_{r/4}(p)$ we have
\begin{equation}
|d(x,y)-d_{\mathbb{R}\times N}(\Xi(x),\Xi(y))|\leq \Psi^{\varsigma/4}.
\end{equation}
The almost splitting \eqref{1567} then follows from the fact that for a rough isometry, the Gromov-Hausdorff distance is bounded above by a multiple of the distortion parameter, namely $\tfrac{3}{2}\Psi^{\varsigma/4}$.

It remains to establish \eqref{poiu}. Observe that the quantities in question arise in the proof of the Hessian bound, Proposition \ref{proposition4.1} $(iii)$. More precisely, they arose from the Bochner identity \eqref{eq4.8} and subsequent integration by parts \eqref{eq4.11}, both with advantageous signs. By keeping these terms in all subsequent estimates instead of discarding them, the desired result follows.
\end{proof}

An immediate consequence of the almost splitting theorem asserts that the splitting
extends to limit metric spaces under Gromov-Hausdorff convergence. The proof requires no further modifications in the current setting and may be found in
\cite[page 23]{WZ}.

\begin{corollary}\label{corsequence}
Let $(M_i,g_i,X_i)$ be a sequence of complete Riemannian manifolds of dimension $n$ with smooth vector fields $X_i$. Assume that $\mathrm{Ric}_{X_i}^m(g_i)\ge -(n-1)\delta_i g_i$ with $\delta_i\rightarrow 0$, $\sup_{M_i}\left(|X_i|+|\mathrm{div}X_i|\right)\leq\mathcal{C}$, and $(M_i,p_i)\rightarrow (M_{\infty},p_{\infty})$ in the pointed Gromov-Hausdorff sense. If $M_{\infty}$ contains a line passing through $p_{\infty}$, then
$M_{\infty}=\mathbb{R}\times N$ for some length space $N$.
\end{corollary}

\section{Topological Consequences of Almost Splitting}
\label{sec7} \setcounter{equation}{0}
\setcounter{section}{7}

The goal of this section is to establish the almost abelian characterization
of fundamental groups arising from manifolds admitting almost nonnegative Bakry-\'Emery Ricci curvature, that is Theorem \ref{theorem1.2}. Such a result was established by Yun \cite{Yun} for almost nonnegative Ricci curvature, building on work of Wei \cite{Wei} which showed that the associated fundamental groups were of polynomial growth. These conclusions were later extended to the gradient Bakry-\'Emery setting by Jaramillo \cite[Theorem 1.3]{Jaramillo}.

\subsection{A volume estimate}
\label{7.1}

A key ingredient in the proof of the almost abelian characterization of fundamental groups under Ricci curvature bounds is the Bishop-Gromov inequality, in particular when the smaller radius tends to zero. Although analogues of the Bishop-Gromov inequality continue to hold in the Bakry-\'Emery case, such as Corollary \ref{corollary2.6}, the disparity of dimension exhibited by the model comparison space renders this inequality useless when sending the smaller radius to zero. To overcome this problem, an alternative volume estimate must be established. In the gradient case \cite[Proposition 3.2]{Jaramillo} this follows from relatively straightforward manipulations of the mean curvature comparison proof. Surprisingly, the non-gradient setting is somewhat more difficult to deal with, and requires a finely tuned choice of stand-in for the potential function, $f$. Furthermore, the estimate obtained below differs significantly from that of \cite{Jaramillo} by an extra factor of polynomial growth determined by the \textit{synthetic dimension} $m+n$. In what follows, the weighted $f$-volume will be denoted by $\mathrm{Vol}_{f}(B_r)=\int_{B_r}e^{-f}dV_g$.

\begin{prop}\label{volest}
Consider a compact Riemannian manifold $(M,g,X)$ of dimension $n$ with smooth 1-form $X$, with Riemannian cover $(\tilde{M},\tilde{g},\tilde{X})$. Let $m>0$, $\delta\ge 0$, and assume that $\mathrm{Ric}_X^m(g)\ge -(n-1)\delta g$ as well as $\sup_{M}|X|\leq\mathcal{C}$.  Then there exists an analytic function $h$ on $\mathbb{R}_+$ depending only on $m$ and $\mathcal{C}$ which is nondecreasing with $h(0)=0$, and a smooth potential function $f$ on $M$ with pullback $\tilde{f}$ on $\tilde{M}$, such that
\begin{equation}\label{jui}
\mathrm{Vol}_{\tilde{f}}(\tilde{B}_r)\leq (r+1)^m e^{\left[ \sqrt{\delta}(r^2+r^3)h(\sqrt{\delta}r)+\mathcal{C}\right]}
\overline{\mathrm{Vol}}_n(B_r)
\end{equation}
for all $r\ge 0$, where $\tilde{B}_r\subset\tilde{M}$ is a geodesic ball, and $\overline{\mathrm{Vol}}_n(B_r)$ is the volume of a geodesic ball in a comparison space of dimension $n$ with curvature $-\delta$.
\end{prop}


\begin{proof}
From \eqref{eq2.14} with $d=n$ we have
\begin{equation}\label{====}
\partial_{\rho}\left(\log\tilde{\mathcal{A}}\right)
\leq\partial_{\rho}\left(\log\bar{\mathcal{A}}\right)+\langle\tilde{X},\partial_{\rho}\rangle
-\ell^{-2}(r)\int_{0}^{r}\left[\left(\ell^2\right)'
\langle\tilde{X},\partial_{\rho}\rangle+\frac{1}{m}\ell^2 \langle\tilde{X},\partial_{\rho}\rangle^2\right]d\rho,
\end{equation}
where $\tilde{\mathcal{A}}$ and $\bar{\mathcal{A}}$ define volume forms for geodesic spheres in $\tilde{M}$ and the model space as in \eqref{677}. Recall that $\tilde{\mathcal{A}}$ is defined on the star-shaped segment domain (interior) $\mathrm{seg}^0(p)\subset T_{p}\tilde{M}$ where $\exp_p$ is injective \cite{Petersen}. In order to extend this to the whole tangent space, let $\varphi_{\epsilon}\in C^{\infty}(T_{p}\tilde{M})$ be a nonnegative cut-off function such that in polar coordinates
\begin{equation}
 \varphi_{\epsilon}(\rho,\theta)=
    \begin{cases}
      1 & 0\leq \rho<\rho(\theta)-\epsilon \\
      0 & \rho\geq \rho(\theta)
    \end{cases},
\end{equation}
where $(\rho(\theta),\theta)\in\partial\mathrm{seg}(p)$ or $\rho(\theta)=\infty$ (if there is no cut point along this direction), and with the property that $\partial_{\rho}\varphi_{\epsilon}\leq 0$ as well as $0\leq\varphi_{\epsilon}\leq 1$. Note that the distance to the cut locus along radial lines, $\rho:S^{n-1}\rightarrow\mathbb{R}\cup\{\infty\}$, is a continuous function \cite[Proposition 13.2.9]{DoCarmo}.
Let $f$ be a smooth function on $M$ to be chosen later, with pullback $\tilde{f}$. Observe that
\begin{equation}\label{]]}
\partial_{\rho}(\varphi_{\epsilon} e^{-\tilde{f}}\tilde{\mathcal{A}})
= e^{-\tilde{f}}\tilde{\mathcal{A}}\partial_{\rho}\varphi_{\epsilon}
+\varphi_{\epsilon}\partial_{\rho}( e^{-\tilde{f}}\tilde{\mathcal{A}})
\leq \varphi_{\epsilon}\partial_{\rho}( e^{-\tilde{f}}\tilde{\mathcal{A}}).
\end{equation}
Therefore multiplying \eqref{====} through by $\varphi_{\epsilon}e^{-\tilde{f}}\tilde{\mathcal{A}}$, using \eqref{]]}, and integrating over the sphere produces
\begin{align}
\begin{split}
\int_{S^{n-1}}\partial_{\rho}\left(\varphi_{\epsilon}
e^{-\tilde{f}}\tilde{\mathcal{A}}\right)
\leq&\partial_{\rho}\left(\log\bar{\mathcal{A}}\right)
\int_{S^{n-1}}\varphi_{\epsilon}e^{-\tilde{f}}\tilde{\mathcal{A}}
+\int_{S^{n-1}}\left(\langle\tilde{X},\partial_{\rho}\rangle -\partial_{\rho}\tilde{f}\right)\varphi_{\epsilon}
e^{-\tilde{f}}\tilde{\mathcal{A}}\\
&-\int_{S^{n-1}}\varphi_{\epsilon}e^{-\tilde{f}}\tilde{\mathcal{A}}
\ell^{-2}(r)\int_{0}^{r}\left[\left(\ell^2\right)'\langle\tilde{X},\partial_{\rho}\rangle +\frac{1}{m}\ell^2 \langle\tilde{X},\partial_{\rho}\rangle^2\right]d\rho.
\end{split}
\end{align}
Next, divide by $\int_{S^{n-1}}\varphi_{\epsilon}e^{-\tilde{f}}\tilde{\mathcal{A}}$ and integrate from $r_1>0$ to $r_2$, staying within the range where this integral is nonzero, to obtain
\begin{align}\label{6--}
\begin{split}
&\int_{S^{n-1}}\varphi_{\epsilon}e^{-\tilde{f}}\tilde{\mathcal{A}}(r_2)\\
\leq&\left(\int_{S^{n-1}}\bar{\mathcal{A}}(r_2)\right)
\left(\frac{\int_{S^{n-1}}\varphi_{\epsilon}e^{-\tilde{f}}\tilde{\mathcal{A}}(r_1)}
{\int_{S^{n-1}}\bar{\mathcal{A}}(r_1)}\right)
\exp\left\{\int_{r_1}^{r_2}
\left(\fint_{S^{n-1}}\left(\langle\tilde{X},\partial_{\rho}\rangle -\partial_{\rho}\tilde{f}\right)\varphi_{\epsilon}e^{-\tilde{f}}\tilde{\mathcal{A}}
\right)dr\right\}\\
&\cdot\exp\left\{-\int_{r_1}^{r_2}
\left(\fint_{S^{n-1}}\varphi_{\epsilon}e^{-\tilde{f}}\tilde{\mathcal{A}}
\ell^{-2}(r)\int_{0}^{r}\left[\left(\ell^2\right)'
\langle\tilde{X},\partial_{\rho}\rangle+\frac{1}{m}\ell^2 \langle\tilde{X},\partial_{\rho}\rangle^2\right] d\rho
\right)dr\right\},
\end{split}
\end{align}
where $\fint_{S^{n-1}}$ indicates the average value with respect to the measure defined by $\varphi_{\epsilon}e^{-\tilde{f}}\tilde{\mathcal{A}}$. Then sending $\epsilon, r_1\rightarrow 0$ produces
\begin{align}\label{mty4}
\begin{split}
&\int_{S^{n-1}}e^{-\tilde{f}}\tilde{\mathcal{A}}_0(r_2)\\
\leq&e^{-\tilde{f}(0)}\left(\int_{S^{n-1}}\bar{\mathcal{A}}(r_2)\right)
\exp\left\{\int_{0}^{r_2}
\left(\fint_{S^{n-1}}\left(\langle\tilde{X},\partial_{\rho}\rangle -\partial_{\rho}\tilde{f}\right)e^{-\tilde{f}}\tilde{\mathcal{A}}_0
\right)dr\right\}\\
&\cdot\exp\left\{-\int_{0}^{r_2}
\left(\fint_{S^{n-1}}e^{-\tilde{f}}\tilde{\mathcal{A}}_0
\ell^{-2}(r)\int_{0}^{r}\left[\left(\ell^2\right)'
\langle\tilde{X},\partial_{\rho}\rangle+\frac{1}{m}\ell^2 \langle\tilde{X},\partial_{\rho}\rangle^2\right] d\rho
\right)dr\right\},
\end{split}
\end{align}
where $\tilde{\mathcal{A}}_0$ agrees with $\tilde{\mathcal{A}}$ on the segment domain interior and it vanishes on the complement $T_{p}\tilde{M}\setminus\mathrm{seg}^0(p)$.

Let us now estimate the last term on the right-hand side of \eqref{mty4}.
First observe that
\begin{equation}
\ell^2=\delta^{-1}\sinh^2(\sqrt{\delta}r)=r^2 +O(\delta r^4),\quad\quad\quad
(\ell^2)'=\delta^{-1/2}\sinh(2\sqrt{\delta}r)=2r+O(\delta r^3).
\end{equation}
We then have
\begin{align}
\begin{split}
&-\fint_{S^{n-1}}e^{-\tilde{f}}\tilde{\mathcal{A}}_0
\ell^{-2}(r)\int_{0}^{r}\left[\left(\ell^2\right)'
\langle\tilde{X},\partial_{\rho}\rangle+\frac{1}{m}\ell^2 \langle\tilde{X},\partial_{\rho}\rangle^2\right] d\rho\\
\leq&\fint_{S^{n-1}}\frac{e^{-\tilde{f}}\tilde{\mathcal{A}}_0}{
\sinh^2(\sqrt{\delta}r)}\int_{0}^{r}\left[2\delta \rho
|\langle\tilde{X},\partial_{\rho}\rangle|-\frac{\delta \rho^2}{m} \langle\tilde{X},\partial_{\rho}\rangle^2\right] d\rho\\
&+\frac{\sqrt{\delta}\mathcal{C}}{\sinh^2(\sqrt{\delta}r)}
\int_{0}^{r}\left(\sinh(2\sqrt{\delta}\rho)-2\sqrt{\delta}\rho\right)d\rho\\
&\leq\fint_{S^{n-1}}e^{-\tilde{f}}\tilde{\mathcal{A}}_0r^{-2}
\int_{0}^{r}\left[2 \rho
|\langle\tilde{X},\partial_{\rho}\rangle|-\frac{\rho^2}{m} \langle\tilde{X},\partial_{\rho}\rangle^2\right] d\rho\\
&+\frac{\mathcal{C}\left(\cosh(2\sqrt{\delta}r)-1-2\delta r^2\right)}{2\sinh^2(\sqrt{\delta}r)}
+\frac{\mathcal{C}^2 r^3}{3m}\left(\frac{1}{r^2}-\frac{\delta}{\sinh^2(\sqrt{\delta}r)}
\right)\\
=&\fint_{S^{n-1}}e^{-\tilde{f}}\tilde{\mathcal{A}}_0r^{-2}
\int_{0}^{r}\left[2 \rho
|\langle\tilde{X},\partial_{\rho}\rangle|-\frac{\rho^2}{m} \langle\tilde{X},\partial_{\rho}\rangle^2\right] d\rho
+\left(\mathcal{C}+\frac{\mathcal{C}^2 r}{3m}\right)
\left(1-\frac{\delta r^2}{\sinh^2(\sqrt{\delta}r)}\right),
\end{split}
\end{align}
and it follows that
\begin{align}
\begin{split}
&-\int_{0}^{r_2}\left(\fint_{S^{n-1}}e^{-\tilde{f}}\tilde{\mathcal{A}}_0
\ell^{-2}(r)\int_{0}^{r}\left[\left(\ell^2\right)'
\langle\tilde{X},\partial_{\rho}\rangle+\frac{1}{m}\ell^2 \langle\tilde{X},\partial_{\rho}\rangle^2\right] d\rho
\right)dr\\
\leq & \int_{0}^{r_2}\left(\fint_{S^{n-1}}e^{-\tilde{f}}\tilde{\mathcal{A}}_0 r^{-2}
\int_{0}^{r}\left[2 \rho
|\langle\tilde{X},\partial_{\rho}\rangle|-\frac{\rho^2}{m} \langle\tilde{X},\partial_{\rho}\rangle^2\right] d\rho\right)dr
+\sqrt{\delta}(r^2_2+r^3_2)h(\sqrt{\delta}r_2)
\end{split}
\end{align}
where $h$ is an analytic function satisfying the desired properties. To see this last part, note that
\begin{align}
\begin{split}
\frac{1}{r^2}-\frac{\delta}{\sinh^2(\sqrt{\delta}r)}=&
\frac{\sinh^2(\sqrt{\delta}r)-\delta r^2}{r^2 \sinh^2(\sqrt{\delta}r)}\\
=&\frac{\left[\sqrt{\delta}r+\tfrac{1}{3!}(\sqrt{\delta}r)^3+O(\sqrt{\delta}r)^5\right]^2
-\delta r^2}{r^2 \left[\sqrt{\delta}r+\tfrac{1}{3!}(\sqrt{\delta}r)^3+O(\sqrt{\delta}r)^5\right]^2}\\
=&\frac{\delta}{3}\left[\frac{1+O(\sqrt{\delta}r)^2}{1+O(\sqrt{\delta}r)^2}\right]\\
=:&\frac{\delta}{3}\bar{h}(\sqrt{\delta}r),
\end{split}
\end{align}
where $\bar{h}$ is a positive analytic function with $\bar{h}(0)=1$. Therefore
\begin{align}
\begin{split}
\int_{0}^{r_2}\left(\mathcal{C}+\frac{\mathcal{C}^2 r}{3m}\right)\left(1-\frac{\delta r^2}{\sinh^2(\sqrt{\delta}r)}\right)dr
\leq& \left(\mathcal{C}r_{2}^2+\frac{\mathcal{C}^2 r_{2}^3}{3m}\right)
\int_{0}^{r_2}\left(\frac{1}{r^2}-\frac{\delta}{\sinh^2(\sqrt{\delta}r)}\right)dr\\
=&\left(\mathcal{C}r_{2}^2+\frac{\mathcal{C}^2 r_{2}^3}{3m}\right)
\int_{0}^{r_2}\frac{\delta}{3}\bar{h}(\sqrt{\delta}r)dr\\
=&\frac{\sqrt{\delta}}{3}\left(\mathcal{C}r_{2}^2+\frac{\mathcal{C}^2 r_{2}^3}{3m}\right)
\int_{0}^{\sqrt{\delta}r_2}\bar{h}(\bar{r})d\bar{r}\\
\leq&\sqrt{\delta}(r^2_2+r^3_2)h(\sqrt{\delta}r_2),
\end{split}
\end{align}
where
\begin{equation}
h(r):=\max\left\{\frac{\mathcal{C}}{3},\frac{\mathcal{C}^2}{9m}\right\}
\int_{0}^{r}\bar{h}(\bar{r})d\bar{r}.
\end{equation}

Next apply Jensen's inequality to obtain
\begin{align}
\begin{split}
&\fint_{S^{n-1}}e^{-\tilde{f}}\tilde{\mathcal{A}}_0 r^{-2}
\int_{0}^{r}\left[2 \rho
|\langle\tilde{X},\partial_{\rho}\rangle|-\frac{\rho^2}{m} \langle\tilde{X},\partial_{\rho}\rangle^2\right] d\rho\\
\leq &\frac{1}{r^2}\int_{0}^{r}\left[2 \left(\fint_{S^{n-1}}\rho e^{-\tilde{f}}\tilde{\mathcal{A}}_0(r)
|\langle\tilde{X},\partial_{\rho}\rangle|(\rho)\right)-\frac{1}{m}
\left(\fint_{S^{n-1}}\rho e^{-\tilde{f}}\tilde{\mathcal{A}}_0(r)
|\langle\tilde{X},\partial_{\rho}\rangle|(\rho)\right)^2\right] d\rho\\
\leq &\frac{1}{r}\left[2 \left(\fint_{0}^{r}\fint_{S^{n-1}}\rho e^{-\tilde{f}}\tilde{\mathcal{A}}_0(r)
|\langle\tilde{X},\partial_{\rho}\rangle|(\rho)d\rho\right)-\frac{1}{m}
\left(\fint_{0}^r\fint_{S^{n-1}}\rho e^{-\tilde{f}}\tilde{\mathcal{A}}_0(r)
|\langle\tilde{X},\partial_{\rho}\rangle|(\rho)d\rho\right)^2\right]\\
=&r^{-1}\left[2a(r)-m^{-1}a(r)^2\right],
\end{split}
\end{align}
where
\begin{equation}
a(r)=\fint_{0}^{r}\fint_{S^{n-1}}\rho e^{-\tilde{f}}\tilde{\mathcal{A}}_0(r)
|\langle\tilde{X},\partial_{\rho}\rangle|(\rho)d\rho.
\end{equation}
Since $2a(r)-m^{-1}a(r)^2\leq m$ for $r\geq 0$ and $2a(r)-m^{-1}a(r)^2\leq\mathcal{C}r$ for $0\leq r\leq 1$, we find that
\begin{align}\label{oooo}
\begin{split}
&-\int_{0}^{r_2}\left(\fint_{S^{n-1}}e^{-\tilde{f}}\tilde{\mathcal{A}}_0
\ell^{-2}(r)\int_{0}^{r}\left[\left(\ell^2\right)'
\langle\tilde{X},\partial_{\rho}\rangle+\frac{1}{m}\ell^2 \langle\tilde{X},\partial_{\rho}\rangle^2\right] d\rho
\right)dr\\
\leq& m\log (r_2+1)+\mathcal{C}+\sqrt{\delta}(r^2_2+r^3_2)h(\sqrt{\delta}r_2).
\end{split}
\end{align}

Consider now the second term on the right-hand side of \eqref{mty4}. According to the Hodge decomposition \cite{Warner}, on $M$ there exists a harmonic 1-form $\omega$, a function $\alpha$, and 2-form $\beta$ such that
\begin{equation}
e^{-f}\left(X-df\right)=\omega+d\alpha+d^{*}\beta,
\end{equation}
where $d^{*}$ denotes the $L^2$ adjoint of the exterior derivative $d$. In particular $\omega+d^* \beta$ is divergence free so that
\begin{equation}
d^*(\omega+d^*\beta)=0,\quad\quad\quad\quad -\Delta\alpha=d^*\left[e^{-f}\left(X-df\right)\right].
\end{equation}
For $u\in C^{\infty}(M)$ set $Lu=\Delta u +\mathrm{div}(uX)$ and note that
\begin{equation}\label{6789}
Le^{-f}=-d^*\left[e^{-f}\left(X-df\right)\right].
\end{equation}
We claim that there exists a positive function $u_0$ on $M$ satisfying $Lu_0=0$. To see this observe that the adjoint $L^*=\Delta-X\cdot\nabla$ admits a maximum principle, and therefore $\mathrm{Ker} L^*=\{const.\}$. It follows from the Fredholm alternative \cite{Evans} that $\mathrm{dim}\mathrm{Ker}L=1$, and so there is $0\neq u_0\in \mathrm{Ker}L$. It remains to show that $u_0$ does not change sign.
In fact, the existence of $u_0$, as well as its positivity follows from Lemma 4.1 of \cite{AnderssonMarsSimon}. Indeed, according to part $(i)$ of this result there exists a real principal eigenvalue $\lambda$ of $L$ with corresponding eigenfunction $u_0>0$, so that $Lu_0=\lambda u_0$. Then integrating this equation over $M$ shows that $\lambda=0$.

Let us now choose $f=-\log u_0$, and scale $u_0$ appropriately to achieve $\tilde{f}(0)=0$. With this selection $\Delta\alpha=0$ so that $\alpha=const.$ This shows that $e^{-f}\left(X-df\right)=\omega+d^*\beta$ is divergence free, and hence by the divergence theorem
\begin{equation}\label{jjjj}
\int_{S^{n-1}}\left(\langle\tilde{X},\partial_{\rho}\rangle -\partial_{\rho}\tilde{f}\right)e^{-\tilde{f}}\tilde{\mathcal{A}}_0 (r)
=\int_{\partial\tilde{B}_r}(\tilde{\omega}+d^*\tilde{\beta})(\partial_{\rho})=0
\end{equation}
for almost every $r$, where $\tilde{\omega}$ and $\tilde{\beta}$ are the pullback forms on $\tilde{M}$. We remark that $\tilde{B}_r$ may not have smooth boundary due to the cut locus, which has measure zero \cite{ItohTanaka}. However, since it is defined via the level set of a (positive) Lipschitz function, it is a set of locally finite perimeter for almost every $r$ \cite[Example 13.3]{Maggi}, \cite[Proposition 5.7.5]{Pfeffer}. Moreover, the divergence theorem holds for regular forms (or vector fields) on such sets \cite[Theorem 6.5.4]{Pfeffer}. It should be pointed out that the boundary term in the divergence theorem should be computed with respect to the $(n-1)$-dimensional Hausdorff measure $\mathcal{H}^{n-1}$ of the reduced boundary, however a consequence of the coarea formula and De Giorgi's structure theorem shows the $\mathcal{H}^{n-1}$-equivalence of the topological and reduced boundaries of a.e. level set of a Lipschitz function \cite[Remark 18.2]{Maggi}.
Combining \eqref{mty4}, \eqref{oooo}, \eqref{jjjj}, and taking an exponential produces
\begin{equation}
\int_{S^{n-1}}e^{-\tilde{f}}\tilde{\mathcal{A}}_0(r_2)
\leq
(r_2 +1)^me^{\left[\sqrt{\delta}(r^2_2+r^3_2)h(\sqrt{\delta}r_2)+\mathcal{C}\right]}
\int_{S^{n-1}}\bar{\mathcal{A}}(r_2).
\end{equation}
Finally, integrating $r_2$ over the interval $[0,r]$ yields the desired conclusion.
\end{proof}

\begin{remark}\label{remark91}
Let $r_1<r_2$ be as in the above proof. If $r_1$ is not sent to zero after \eqref{6--}, then with suitable modifications of the arguments we obtain a variant of the Bishop-Gromov inequality
\begin{equation}
\frac{\mathrm{Vol}_{\tilde{f}}(\tilde{B}_{r_2})}
{\mathrm{Vol}_{\tilde{f}}(\tilde{B}_{r_1})}
\leq \frac{\int_{0}^{r_2}(\rho+1)^m e^{\left[\sqrt{\delta}(\rho^2+\rho^3)h(\sqrt{\delta}\rho)
+\mathcal{C}_0\right]}\ell^{n-1} d\rho}
{\int_{0}^{r_1}(\rho+1)^m e^{\left[\sqrt{\delta}(\rho^2+\rho^3)h(\sqrt{\delta}\rho)\right]}\ell^{n-1} d\rho},
\end{equation}
where $\mathcal{C}_0$ is a constant depending on $m$ and $\mathcal{C}$.
\end{remark}

\subsection{Applications of the volume estimate}

The volume estimate of the previous section may be used to generalize results of Anderson \cite{Anderson} concerning the structure of fundamental groups under Ricci curvature lower bounds, as well as a polynomial growth characterization of Wei \cite{Wei}, all of which are used in the desired almost abelian result. Generalizations to the gradient Bakry-\'Emery setting were given by Jaramillo in \cite{Jaramillo}. The proofs follow in the nongradient setting in a similar way. However, due to the difference in growth in the volume estimate between the gradient and nongradient cases, we retain an outline of the arguments where appropriate to indicate the required modifications.

\begin{lemma}\label{1234567}
Consider a complete Riemannian manifold $(M,g,X)$ of dimension $n$ with smooth 1-form $X$. Let $m>0$, $\delta\ge 0$, and assume that
\begin{equation}\label{78iu}
\mathrm{Ric}_X^m(g)\ge -(n-1)\delta g, \quad\quad\mathrm{diam}(M)\leq\mathcal{D},
\quad\quad
\mathrm{Vol}_f(M)\geq\mathcal{V}, \quad\quad \sup_{M}|X|\leq\mathcal{C}.
\end{equation}
If $\Gamma\leq\pi_1(M)$ is a subgroup generated by loops $\gamma_i$, $i=1,\ldots,k$ with $k\geq N$, then the maximum generator length satisfies $\max_i l(\gamma_i)\geq\mathcal{D}/N$ where
\begin{equation}\label{ghj}
N=\mathcal{V}^{-1}(2\mathcal{D}+1)^m e^{\left[ \sqrt{\delta}((2\mathcal{D})^2+(2\mathcal{D})^3)
h(\sqrt{\delta}2\mathcal{D})+\mathcal{C}\right]}
\overline{\mathrm{Vol}}_n(B_{2\mathcal{D}}),
\end{equation}
and $f$ and $h$ are given in Proposition \ref{volest}. Furthermore, among the class of manifolds satisfying \eqref{78iu} there are only finitely many isomorphism types of $\pi_1(M)$.
\end{lemma}

\begin{remark}
Analogous statements hold if the hypothesis $\mathrm{Vol}_f (M)\geq\mathcal{V}$ is replaced by $\mathrm{Vol}(M)\geq\mathcal{V}$, in light of the proof of Theorem \ref{abelian} below.
\end{remark}

\begin{proof}
Following \cite[Theorem 2.1, Remark 2.2(2)]{Anderson} define
\begin{equation}
U(J)=\left\{\mathrm{g}\in\Gamma\mid\mathrm{g}=[\gamma_1]^{j_1} \cdots[\gamma_{a_0}]^{j_{a_0}},\text{ }\sum|j_a|\leq J\right\},
\end{equation}
and choose the smallest $J_0$ such that $\# U(J_0)>N$. If $F\subset\tilde{M}$ is a fundamental domain of the universal cover, which contains $\tilde{x}_0$ lying in the preimage of $x_0$, and $r_0=N\max_i l(\gamma_i)+\mathcal{D}$ then
\begin{equation}
\bigcup_{\mathrm{g}\in U(J_0)}\mathrm{g}\left(\tilde{B}_{\mathcal{D}}(\tilde{x}_0)\cap F\right)\subset \tilde{B}_{r_0}(\tilde{x}_0).
\end{equation}
It follows that the volume estimate of Proposition \ref{volest} implies
\begin{align}\label{qw}
\begin{split}
N\mathcal{V}\leq N\mathrm{Vol}_{f}(M)\leq&
\mathrm{Vol}_{\tilde{f}}\left(\tilde{B}_{r_0}(\tilde{x}_0)\right)\\
\leq&
(r_0+1)^m e^{\left[\sqrt{\delta}(r^2_0+r^3_0)
h(\sqrt{\delta}r_0)+\mathcal{C}\right]}
\overline{\mathrm{Vol}}_n(B_{r_0}).
\end{split}
\end{align}
If it were the case that $\max_i l(\gamma_i)<\mathcal{D}/N$ then \eqref{qw} yields
\begin{equation}
N<\mathcal{V}^{-1}(2\mathcal{D}+1)^m e^{\left[ \sqrt{\delta}((2\mathcal{D})^2+(2\mathcal{D})^3)
h(\sqrt{\delta}2\mathcal{D})+\mathcal{C}\right]}
\overline{\mathrm{Vol}}_n(B_{2\mathcal{D}}),
\end{equation}
a contradiction. Therefore $\max_i l(\gamma_i)\geq\mathcal{D}/N$.
Moreover, as in \cite[Theorem 2.3]{Anderson}, the finite number of isomorphism types of $\pi_1(M)$ follows from the above loop inequality and Proposition \ref{volest}, as well as a result of Gromov \cite[Proposition 5.28]{Gromov} concerning generators of the fundamental group.
\end{proof}

We are now able to establish a polynomial growth result for the fundamental group, generalizing \cite[Theorem 3.5]{Jaramillo} and \cite[Theorem 1]{Wei}.

\begin{lemma}\label{polynomial}
Consider a complete Riemannian manifold $(M,g,X)$ of dimension $n$ with smooth 1-form $X$. Let $m>0$, $\delta\ge 0$, and assume that
\begin{equation}\label{vgh}
\mathrm{Ric}_X^m(g)\ge -(n-1)\delta g,\quad\quad
\mathrm{diam}(M)\leq\mathcal{D},
\quad\quad
\mathrm{Vol}_f(M)\geq\mathcal{V}, \quad\quad \sup_{M}|X|\leq\mathcal{C},
\end{equation}
where $f$ is given in Proposition \ref{volest}. There exists $\delta_0\left(n,m,\mathcal{C},\mathcal{D},\mathcal{V}\right)>0$, such that if $\delta\leq\delta_0$ then $\pi_1(M)$ is of polynomial growth of degree $\leq n+m$.
\end{lemma}

\begin{proof}
Assume that the conclusion is false. Then there exists a sequence of manifolds $(M_i,g_i,X_i)$, and constants $\delta_i\rightarrow 0$, satisfying \eqref{vgh} such that $\pi_1(M_i)$ is not of polynomial growth of degree $\leq n+m$. Therefore, if $\Gamma_i(s)$ denotes the set of distinct words in $\pi_1(M_i)$ of length $\leq s$, then for any set of generators of $\pi_1(M_i)$ we can find $s_i\rightarrow\infty$ such that
\begin{equation}\label{numbergen}
\#\Gamma_i(s_i)>i s_i^{n+m},\quad\quad\quad \sqrt{\delta_i}s_i^3\rightarrow 0.
\end{equation}
This is achieved using the freedom to choose $s_i$ along with the following observation.
Lemma \ref{1234567} states that when \eqref{vgh} is satisfied there are finitely many isomorphism types of $\pi_1(M)$, and according to \cite[Proposition 5.28]{Gromov} for each isomorphism type there are
generating loops $\gamma_j$, $j=1,\ldots,J$ with the property that $\max_j l(\gamma_j)\leq 3\mathcal{D}$ and all relations in these generators are of the form $[\gamma_j][\gamma_k]=[\gamma_l]$.  Note that the control on generator length Lemma \ref{1234567}, together with the proof of \cite[Theorem 2.3]{Anderson} in which Proposition \ref{volest} is used in place of Bishop-Gromov volume comparison, shows that the number of generators $J$ is bounded above in terms of $n$, $m$, $\mathcal{C}$, $\mathcal{D}$, and $\mathcal{V}$. In particular, the number of generators used to describe \eqref{numbergen} may be taken independent of $i$.

Let $\tilde{x}_0^i\in\tilde{M}_i$ be in the preimage (within the universal cover) of a chosen base point $x_0^i\in M_i$ for the fundamental group, and choose a fundamental domain $F_i$ for $\pi_1(M_i)$ containing $\tilde{x}_0^i$. If $r_i=(3s_i +1)\mathcal{D}$ then
\begin{equation}
\bigcup_{\mathrm{g}\in\Gamma_i(s_i)}\mathrm{g}\left(F_i\right)\subset
\tilde{B}_{r_i}(\tilde{x}_0^i).
\end{equation}
We then have, by Proposition \ref{volest}, that for sufficiently large $i$,
\begin{align}
\begin{split}
\#\Gamma_i(s_i)\mathcal{V}\leq&\#\Gamma_i(s_i)\mathrm{Vol}_{f_i}(M_i)\\
\leq & \mathrm{Vol}_{\tilde{f}_i}\left(\tilde{B}_{r_i}(\tilde{x}_0^i)\right)\\
\leq & (r_i+1)^m e^{\left[\sqrt{\delta_i}(r^2_i+r^3_i)
h(\sqrt{\delta_i}r_i)+\mathcal{C}\right]}
\overline{\mathrm{Vol}}_n(B_{r_i})\\
\leq & \left((3s_i +1)\mathcal{D}+1\right)^m e^{\mathcal{C}+1}|S^{n-1}|
\int_{0}^{(3s_i +1)\mathcal{D}}\left(\frac{\sinh(\sqrt{\delta_i} \rho)}{\sqrt{\delta_i}}\right)^{n-1}d\rho\\
\leq & \frac{(4\mathcal{D})^{n+m}e^{\mathcal{C}+1}|S^{n-1}|}{n}s_i^{n+m}.
\end{split}
\end{align}
This, however, contradicts \eqref{numbergen}.
\end{proof}

From the polynomial growth property, Yun \cite[Theorem 2]{Yun} was able to establish the almost abelian characterization of the fundamental group for manifolds with almost nonnegative Ricci curvature, and this was extended to the gradient Bakry-\'{E}mery setting by Jaramillo \cite[Theorem 1.3]{Jaramillo}. Here we generalize these results to the nongradient Bakry-\'{E}mery case. The proof relies on the almost splitting result Theorem \ref{theoremalmostsplit1}, the generator length and isomorphism type bounds Lemma \ref{1234567}, and the polynomial growth characterization Lemma \ref{polynomial}. With these ingredients, the arguments of \cite{Yun} apply without change to yield desired almost abelian theorem.

\begin{theorem}\label{abelian}
Consider a complete Riemannian manifold $(M,g,X)$ of dimension $n$ with smooth 1-form $X$. Let $m>0$, $\delta\ge 0$, and assume that
\begin{equation}
\mathrm{Ric}_X^m(g)\ge -(n-1)\delta g,\quad
\mathrm{diam}(M)\leq\mathcal{D},
\quad
\mathrm{Vol}(M)\geq\mathcal{V}, \quad \sup_{M}\left(|X|+|\nabla\mathrm{div}X|\right)\leq\mathcal{C}.
\end{equation}
There exists $\delta_0\left(n,m,\mathcal{C},\mathcal{D},\mathcal{V}\right)>0$, such that if $\delta\leq\delta_0$ then $\pi_1(M)$ is almost abelian.
\end{theorem}

\begin{proof}
As described above, this follows from the arguments of \cite{Yun} and the previous results of this section. It remains to show that $\mathrm{Vol}_f(M)\geq\mathcal{V}'$ so that these results may be applied, where $f$ is given in Proposition \ref{volest} and $\mathcal{V}'$ depends on $n$, $m$, $\mathcal{C}$, $\mathcal{D}$, and $\mathcal{V}$. To obtain the desired conclusion we will establish a lower bound for $u=e^{-f}$. Recall from \eqref{6789} that
\begin{equation}
\Delta_{-X}u=-\left(\mathrm{div}X\right) u.
\end{equation}
A version of the Cheng-Yau gradient estimate presented in Lemma \ref{lemmaA.1} yields
\begin{equation}
\sup_{M}|\nabla\log u|\leq C_1(n,m,\delta,\mathcal{C}),
\end{equation}
where the bound on $|\nabla\mathrm{div}X|$ is used. Note that an upper bound on the range of $\delta$ allows for a choice of $C_1$ independent of $\delta$. Furthermore, by construction there is a point $x_0\in M$ such that $u(x_0)=1$. Thus, if $\gamma(r)$ is a unit speed minimizing geodesic connecting $x$ to $x_0$ then
\begin{equation}
|\log u(x)|=\left|\int_{0}^{d(x,x_0)} \partial_{r}\log u(\gamma(r))dr\right|
\leq\int_{0}^{d(x,x_0)}|\nabla\log u(\gamma(r))|dr\leq C_1\mathcal{D}.
\end{equation}
It follows that
\begin{equation}
e^{-C_1 \mathcal{D}}\leq u(x)\leq e^{C_1 \mathcal{D}},
\end{equation}
and therefore
\begin{equation}
\mathrm{Vol}_f(M)=\int_{M}e^{-f}dV_g\geq (\inf_{M}u)\mathrm{Vol}(M)\geq e^{-C_1 \mathcal{D}}\mathcal{V}=:\mathcal{V}'.
\end{equation}
\end{proof}

\subsection{A Betti number bound}

The volume estimate of Section \ref{7.1} may be used to obtain a first Betti number bound, generalizing the result of Gallot \cite{Gallot} and Gromov \cite[Theorem 5.21]{Gromov} in the setting of Ricci curvature lower bounds (see also \cite[Theorem 63]{Petersen}). Interestingly the bound we obtain in the Bakry-\'{E}mery setting depends on the synthetic dimension for general $X$, and agrees with the classical result when $X$ is a gradient.

\begin{theorem}\label{betti}
Consider a complete Riemannian manifold $(M,g,X)$ of dimension $n$ with smooth 1-form $X$. Let $m>0$, $\delta\ge 0$, and assume that
\begin{equation}
\mathrm{Ric}_X^m(g)\ge -(n-1)\delta g, \quad\quad\quad\mathrm{diam}(M)\leq\mathcal{D}, \quad\quad\quad \sup_{M}|X|\leq\mathcal{C}.
\end{equation}
Then there is a function $B\left(\delta,n,m,\mathcal{C},\mathcal{D}\right)$ that yields a bound for the first Betti number and satisfies
\begin{equation}
b_1(M)\leq B\left(\delta,n,m,\mathcal{C},\mathcal{D}\right),\quad\quad
\lim_{\delta\rightarrow 0}B\left(\delta,n,m,\mathcal{C},\mathcal{D}\right)=n+m.
\end{equation}
More precisely, there is a $\delta_0(n,m,\mathcal{C},\mathcal{D})>0$ such that if $\delta\leq\delta_0$ then $b_1(M)\leq n+m$.
Furthermore, if $X=df_0$ for some $f_0\in C^{\infty}(M)$ and the assumption $\sup_M |X|\leq\mathcal{C}$ is replaced by $\sup_M |f_0|\leq \mathcal{C}$, then the same conclusions hold with $n+m$ replaced by $n$.
\end{theorem}

\begin{proof}
Recall that $b_1(M)=\mathrm{dim}H_1(M,\mathbb{R})$, and the first homology group is isomorphic to the abelianized fundamental group $H_1(M,\mathbb{Z})=\pi_1(M)/[\pi_1(M),\pi_1(M)]$. This is a finitely generated abelian group, and its torsion subgroup $\mathcal{T}$ is normal. Let $\tilde{M}$ denote the universal cover. Then we may construct a cover
\begin{equation}
\hat{M}=\left(\tilde{M}/[\pi_1(M),\pi_1(M)]\right)/\mathcal{T},
\end{equation}
on which the torsion-free group $G=H_1(M,\mathbb{Z})/\mathcal{T}$ acts by deck transformations, and with $\mathrm{rank}(G)=b_1(M)$. Note that any finite index subgroup of $G$ also has rank $b_1(M)$. According to \cite[Lemma 5.19]{Gromov}, for fixed $\hat{x}\in\hat{M}$, there is a finite index subgroup $\Gamma\leq G$ generated by loops $\gamma_1,\cdots,\gamma_{b_1}\subset M$ such that
\begin{equation}
d(\hat{x},[\gamma_i](\hat{x}))\leq 2\mathrm{diam}(M),\quad\quad\quad
d(\hat{x},\mathrm{g}(\hat{x}))>\mathrm{diam}(M),\text{ }\mathrm{g}\in\Gamma\setminus\{1\}.
\end{equation}

Consider the set
\begin{equation}
U(r)=\{\mathrm{g}\in\Gamma\mid \mathrm{g}=[\gamma_1]^{j_1}\cdots[\gamma_{b_1}]^{j_{b_{1}}},\text{ }\sum|j_a|\leq r\}.
\end{equation}
Observe that for each $\mathrm{g}\in\Gamma\setminus\{1\}$ the balls $\hat{B}_{r_1}(\mathrm{g}(\hat{x}))$ are disjoint where $r_1=\frac{\mathrm{diam}(M)}{2}$, and
\begin{equation}\label{100}
\hat{B}_{\frac{\mathrm{diam}(M)}{2}}(\mathrm{g}(\hat{x}))\subset
\hat{B}_{r_2}(\hat{x}),\quad\quad\quad \mathrm{g}\in U(r)
\end{equation}
where $r_2=2r\mathrm{diam}(M)+\frac{\mathrm{diam}(M)}{2}$. Let $\hat{f}$ denote the pullback to $\hat{M}$ of the function $f\in C^{\infty}(M)$ given by Proposition \ref{volest}. Since the elements of $\Gamma$ act by isometries, the $\hat{f}$-volumes in \eqref{100} have the same value. From Remark \ref{remark91} it follows that
\begin{align}\label{,./}
\begin{split}
\#U(r)\leq&\frac{\mathrm{Vol}_{\hat{f}}\left(\hat{B}_{r_2}(\hat{x})\right)}
{\mathrm{Vol}_{\hat{f}}\left(\hat{B}_{r_1}(\hat{x})\right)}\\
\leq&\frac{\int_{0}^{r_2}(\rho+1)^{m}
e^{\left[ \sqrt{\delta}(\rho^2+\rho^3)
h(\sqrt{\delta}\rho)+\mathcal{C}_0\right]}\ell^{n-1}d\rho}
{\int_{0}^{r_1}(\rho+1)^{m}
e^{\left[ \sqrt{\delta}(\rho^2+\rho^3)
h(\sqrt{\delta}\rho)\right]}\ell^{n-1}d\rho}\\
\leq&\frac{\int_{0}^{2r\mathcal{D}+\frac{\mathcal{D}}{2}}(\rho+1)^{m}
e^{\left[ \sqrt{\delta}(\rho^2+\rho^3)
h(\sqrt{\delta}\rho)+\mathcal{C}_0\right]}
\sinh^{n-1}(\sqrt{\delta}\rho)d\rho}
{\int_{0}^{\frac{\mathcal{D}}{2}}(\rho+1)^{m}
e^{\left[ \sqrt{\delta}(\rho^2+\rho^3)
h(\sqrt{\delta}\rho)\right]}\sinh^{n-1}(\sqrt{\delta}\rho)d\rho}\\
\leq& 5^{n+m} e^{\mathcal{C}_0 +1} r^{m+n}
\end{split}
\end{align}
for large $r$ and sufficiently small $\delta$, with the latter comparatively small relative to the former. On the other hand, by construction, if $r$ is an integer then $\# U(r)=(2r+1)^{b_1}$. Thus, if $b_1>n+m$ then there is a large integer $r_0=r_0(n,m,\mathcal{C})$ satisfying
\begin{equation}\label{65790}
(2r_0 +1)^{b_1}>5^{n+m} e^{\mathcal{C}_0 +1} r^{m+n}_{0}.
\end{equation}
We may now choose $\delta_0=\delta_0(r_0,n,m,\mathcal{C},\mathcal{D})$ such that \eqref{,./} holds for $\delta\leq\delta_0$ with $r=r_0$. The contradiction between \eqref{,./} and \eqref{65790} yields the desired result.

Lastly, if $X=df_0$ for some $f_0\in C^{\infty}(M)$ and the assumption $\sup_M |X|\leq\mathcal{C}$ is replaced by $\sup_M |f_0|\leq \mathcal{C}$, then the same arguments above may be applied with Proposition \ref{volest} replaced by Proposition 3.2 of \cite{Jaramillo}. The factor $(\rho+1)^m$ will not be present in \eqref{,./}, leading to the same conclusions with $n+m$ replaced by $n$.
\end{proof}

\subsection{Applications to the topology of horizons}

Consider the setting of Theorem \ref{maincor}. Recall that the following equation for the $m$-Bakry-\'{E}mery Ricci tensor is induced upon a horizon cross-section $\mathcal{H}$, namely
\begin{equation}
\mathrm{Ric}_{X}^{m}(g)=\frac{2}{n} \Lambda g+2\kappa\chi,
\end{equation}
where $\Lambda$ is the cosmological constant, $\kappa$ is surface gravity, $\chi_{ij}=\langle\pmb{\nabla}_{\partial_i}U,\partial_j\rangle$ is the null second fundamental form in the $U$ direction (transverse to the horizon), and $X$ is a renormalized piece of the Killing vector $V$. By combining this with results of the previous sections we obtain restrictions on horizon topology. Define $\lambda$ to be a lower or upper bound (depending on the sign of $\kappa$) for the eigenvalues of $\chi$, that is
\begin{equation}
\kappa\lambda=\inf_{x\in\mathcal{H}}\min_{ w\in T_{x}\mathcal{H}\atop |w|=1}\kappa\chi(w,w).
\end{equation}
As before let $\mathcal{C}$, $\mathcal{D}$, and $\mathcal{V}$ be constants such that
\begin{equation}\label{,al}
\mathrm{diam}(\mathcal{H})\leq\mathcal{D},
\quad\quad\quad
\mathrm{Vol}(\mathcal{H})\geq\mathcal{V}, \quad \quad\quad \sup_{\mathcal{H}}\left(|X|+|\nabla\mathrm{div}X|\right)\leq\mathcal{C}.
\end{equation}
The next result then follows directly from Theorems \ref{abelian}, \ref{betti}, and the discussion above.

\begin{theorem}\label{theorem-main}
Let $\mathcal{H}$ be a single component compact horizon cross-section in a stationary vacuum spacetime satisfying $\Lambda+n\kappa\lambda\geq-\delta$ and \eqref{,al}.

\begin{itemize}
\item [(i)] There exists $\delta_0(n,\mathcal{C},\mathcal{D},\mathcal{V})>0$, such that if $\delta\leq\delta_0$ then the fundamental group $\pi_{1}(\mathcal{H})$ contains an abelian subgroup of finite index.\smallskip

\item [(ii)] There exists $\delta_0(n,\mathcal{C},\mathcal{D})>0$, such that if $\delta\leq\delta_0$ then the first Betti number satisfies $b_{1}(\mathcal{H})\leq n+2$. Moreover, if $X=df_0$ for some $f_0\in C^{\infty}(\mathcal{H})$ and the assumption $\sup_{\mathcal{H}}\left(|X|+|\nabla\mathrm{div}X|\right)\leq\mathcal{C}$ is replaced by $\sup_{\mathcal{H}} |f_0|\leq \mathcal{C}$, then $b_{1}(\mathcal{H})\leq n$.\smallskip
\end{itemize}
\end{theorem}

We may now establish Theorem \ref{maincor}.
Indeed, compact manifolds have finitely generated fundamental groups, and as shown by Gromov \cite{GromovG} finitely generated almost abelian groups cannot be of exponential growth. Thus, horizons which fit within the context of Theorem \ref{theorem-main} $(\mathrm{i})$ cannot have fundamental groups of exponential growth. In particular, such horizons cannot arise as a nontrivial connected sum except for a few special cases, see the discussion in Section 3 of \cite{KWW}.
This yields part $(\mathrm{i})$ of Theorem \ref{maincor}. Part $(\mathrm{ii})$ of Theorem \ref{maincor} follows directly from Theorem \ref{theorem-main} $(\mathrm{ii})$. Next, observe that if $\Lambda>0$ and the surface gravity is sufficiently small then the horizon cross-section is of positive Bakry-\'{E}mery Ricci curvature, and this implies via a generalization of Myers Theorem that the fundamental group of the horizon must be finite \cite{KhuriWoolgar1}. This gives Theorem \ref{maincor} $(\mathrm{iii})$. Lastly, Lemma \ref{1234567} implies Theorem \ref{maincor} $(\mathrm{iv})$ and completes the proof.

\appendix

\section{A Cheng-Yau Gradient Estimate}
\label{secA}
\setcounter{equation}{0}

\begin{lemma}\label{lemmaA.1}
Let $(M,g,X)$ be a complete Riemannian manifold of dimension $n$ with smooth vector field $X$.  Let $m,\mathcal{C}>0$, $\delta\ge 0$, and $0<r_1<r_2$, and assume that $\mathrm{Ric}_X^m(g)\ge -(n-1)\delta g$ together with $|X|\leq\mathcal{C}$ on $B_{r_2}(p)$. Suppose that $u\in C^{\infty}(B_{r_2}(p))$ is positive and satisfies
\begin{equation}
\label{eqA.1}
\Delta_X u = a F(u),
\end{equation}
for some functions $a\in C^{\infty}(B_{r_2}(p))$ and $F\in C^{\infty}(\mathbb{R}_+)$. Then there exists a constant $C_0\geq 1$ depending on $n,m,\delta,r_1,r_2,\mathcal{C}$ such that
\begin{equation}
\label{eqA.2}
\sup_{B_{r_1}(p)}|\nabla\log u|^2 \!\leq\! C_0 +\sup_{B_{r_2}(p)}\!\left\{\!8n\left[\left(|a|+|\nabla a|\right)\frac{|F(u)|}{u}\!+\! |aF'(u)| \right]\!+\!
4\!\left (\!\mathcal{C}\!+\!\sqrt{\frac{|F(u)|}{u}}\right )^2 \!\right\}.
\end{equation}
\end{lemma}

\begin{proof}
The proof involves a detailed but straightforward calculation that appears in \cite[Chapter 7]{Cheeger}, which we modify to accommodate the vector field $X$. Using equation \eqref{eqA.1} a direct computation shows that for $v:=\log u$ we obtain
\begin{equation}
\label{eqA.3}
\Delta_X v = -\left \vert \nabla v \right \vert^2 +ae^{-v}F(e^v)=:-\left \vert \nabla v \right \vert^2 +aG(v).
\end{equation}
Next, define $Q:=\phi |\nabla v |^2$ where the nonnegative cut-off function $\phi:B_{r_2}(p)\to [0,1]$ is chosen such that $\phi= 1$ on $B_{r_1}(p)$, $\phi=0$ in a neighborhood of $\partial B_{r_2}(p)$, and $\phi\leq 1$ on $B_{r_2}(p)$. In what follows, calculations will be evaluated at a point $q\in B_{r_2}(p)$ where $Q$ takes its maximum, so terms involving $\nabla Q$ will be dropped or rather the identity $0=|\nabla v|^2 \nabla\phi+\phi\nabla\left(|\nabla v|^2 \right)$ will be implemented.

First observe that
\begin{equation}
\label{A.4}
\Delta_X Q = \frac{Q}{\phi} \Delta_X \phi -\frac{2Q}{\phi^2} |\nabla\phi |^2 +\phi \Delta_X \left ( |\nabla v|^2 \right )\ .
\end{equation}
The last term in this formula may be replaced with help from the Bochner formula \cite[Lemma 4]{KWW}
\begin{align}
\label{eqA.5}
\begin{split}
\Delta_X \left ( |\nabla v|^2 \right ) =& 2|\hess v|^2 +2\mathrm{Ric}_X^m (\nabla v, \nabla v) +2\nabla_{\nabla v}\Delta_X v +\frac{2}{m}\left ( X(v)\right )^2\\
\ge & \frac{2}{n}\left ( \Delta v\right )^2 -\frac{2(n-1)\delta}{\phi} Q+2\nabla_{\nabla v}\Delta_X v +\frac{2}{m}\left ( X(v)\right )^2\ ,
\end{split}
\end{align}
where the Bakry-\'{E}mery Ricci curvature lower bound and the Cauchy-Schwarz inequality were used. Furthermore by \eqref{eqA.3}
\begin{align}
\label{eqA.6}
\begin{split}
\frac{2}{n}\phi\left ( \Delta v\right )^2=& \frac{2}{n}\phi^{-1} \left ( \phi \Delta_X v+\phi X(v)\right )^2\\
=& \frac{2}{n}\left (\phi G(v) +\phi X(v)-Q\right )^2,
\end{split}
\end{align}
and
\begin{align}
\label{eqA.7}
\begin{split}
2\phi \nabla_{\nabla v}\Delta_X v =&  2\phi \nabla_{\nabla v}\left ( aG(v)-|\nabla v |^2 \right )\\
=& 2\phi(\nabla v\cdot\nabla a)G(v)+2aG'(v)Q-2\phi \nabla v \cdot \nabla \left (|\nabla v |^2\right ) \\
=& 2\phi(\nabla v\cdot\nabla a)G(v)+2aG'(v)Q+2|\nabla v |^2 \nabla v \cdot \nabla \phi \\
=& 2\phi(\nabla v\cdot\nabla a)G(v)+2aG'(v)Q+\frac{2}{\phi} Q \nabla v \cdot \nabla \phi \\
\ge & -2|\nabla a||G(v)|\phi^{1/2}Q^{1/2}+2aG'(v)Q-4n \frac{|\nabla \phi |^2}{\phi^2}Q -\frac{1}{4n\phi}Q^2.
\end{split}
\end{align}
Gathering the above expressions produces
\begin{align}
\label{eqA.8}
\begin{split}
\phi \Delta_X Q \ge & Q\Delta_X \phi -(2+4n) \frac{|\nabla \phi |^2}{\phi}Q -2|\nabla a||G(v)|\phi^{3/2}Q^{1/2}+2a\phi G'(v)Q-\frac{1}{4n}Q^2\\
& -2(n-1)\delta\phi Q  +\frac{2}{m}\phi \left ( X(v)\right )^2 +\frac{2}{n}\left ( \phi X(v) +\phi G(v) -Q \right )^2
\end{split}
\end{align}
at $q$, where $Q$ takes its maximum.

Now suppose that $Q(q)\le 2\phi \left ( G(v)+X(v) \right )(q)$, then the definitions of $v$, $G$, and $Q$ yield
\begin{equation}
\label{eqA.9}
|\nabla\log u|^2\le 2 u^{-1}\left ( F(u)+X(u)\right ) \le 2u^{-1}|F(u)|  +2\mathcal{C}|\nabla \log u |\quad\quad\text{at}\quad\quad q.
\end{equation}
It follows that
\begin{equation}
\label{eqA.10}
\sup_{B_{r_2}(p)} Q\le 4\left(\mathcal{C}+\sup_{B_{r_2}(p)} \sqrt{u^{-1}|F(u)|}\right)^2.
\end{equation}
If on the other hand $Q(q)\ge 2\phi \left ( G(v)+X(v) \right )(q)$, then this may be manipulated into the form
\begin{equation}
\label{eqA.11}
\frac{2}{n}\left ( \phi X(v) + \phi G(v) -Q\right )^2 -\frac{1}{4n}Q^2 \ge \frac{1}{4n}Q^2 .
\end{equation}
Inserting this into \eqref{eqA.8} and using that $\Delta_X Q\le 0$ at the maximum point $q$, gives rise to
\begin{equation}
\label{eqA.12}
\frac{1}{4n}Q \le  -\Delta_X \phi +(2+4n) \frac{|\nabla \phi |^2}{\phi}
+2|\nabla a||G(v)|\phi^{3/2}Q^{-1/2}
-2a\phi G'(v) +2(n-1)\delta\phi.
\end{equation}
We may assume that $Q(q)> 1$, otherwise \eqref{eqA.2} is automatically valid since $C_0 \geq 1$. It follows that
\begin{equation}\label{111100}
\sup_{B_{r_2}(p)}Q\leq C_0 +\sup_{B_{r_2}(p)}8n\left[(|a|+|\nabla a|)u^{-1}|F(u)|+|aF'(u)| \right].
\end{equation}
In order to show that the constant $C_0$ depends only on the quantities stated in the lemma, we choose the cut-off function $\phi$ to be a non-increasing function of the distance $\rho$ from $p$, so that as in Corollary \ref{corollary2.4} we have $\Delta_X\phi\ge \bar{\Delta}_{n+m}\phi$.
Note that a modification employing a barrier function produces the same result when $q$ is a cut point (see \cite[page 41]{Cheeger}).

Finally observe that the sequence of elementary inequalities
\begin{equation}
\sup_{B_{r_1}(p)}|\nabla\log u|^2=\sup_{B_{r_1}(p)}Q
\leq\sup_{B_{r_2}(p)}Q,
\end{equation}
together with \eqref{eqA.10} and \eqref{111100} gives the desired result.
\end{proof}

\end{document}